\documentclass{amsart}
\usepackage{amssymb}
\usepackage{epsfig}
\usepackage{epic,eepic}
\usepackage{float}
\usepackage{color}
\usepackage{stackengine}
\usepackage{bbm}
\usepackage{hyperref}
\hypersetup{
    colorlinks=true,
    linkcolor=blue,
    filecolor=magenta,      
    urlcolor=cyan,
}
\allowdisplaybreaks

\DeclareMathAlphabet\mathbfcal{OMS}{cmsy}{b}{n}

\hsize \textwidth      
\advance \hsize by -\marginparwidth
\evensidemargin \oddsidemargin 

\newtheorem{theorem}{Theorem}[section]
\newtheorem{proposition}[theorem]{Proposition}
\newtheorem{lemma}[theorem]{Lemma}

\newtheorem{hypothesis}[theorem]{Assumption}

\theoremstyle{definition}
\newtheorem{definition}[theorem]{Definition}

\theoremstyle{remark}
\newtheorem{remark}[theorem]{Remark}

\newcommand\ubar[1]{\stackunder[1.2pt]{$#1$}{\rule{1.0ex}{.2ex}}}

\newcommand{\cM}{{\mathcal{M}}}

\newcommand{\balpha}{{\boldsymbol{\alpha}}}

\newcommand\bL{\ensuremath{\mathbf{L}}}

\newcommand\bX{\ensuremath{\mathbf{X}}}
\newcommand\bY{\ensuremath{\mathbf{Y}}}
\newcommand\bZ{\ensuremath{\mathbf{Z}}}
\newcommand\bp{\ensuremath{\mathbf{p}}}
\newcommand\br{\ensuremath{\mathbf{r}}}

\newcommand\bcM{\mathbfcal{M}}

\begin{document}
\title{Finite-State Contract Theory with a Principal and a Field of Agents}
\author{Ren\'e Carmona \and Peiqi Wang}

\maketitle

\begin{abstract}
We use the recently developed probabilistic analysis of mean field games with finitely many states in the weak formulation, to set-up a principal / agent contract theory model where the principal faces a large population of agents interacting in a mean field manner. We reduce the problem to the optimal control of dynamics of the McKean-Vlasov type, and we solve this problem explicitly in a special case reminiscent of the linear - quadratic mean field game models. 
The paper concludes with a numerical example demonstrating the power of the results when applied to a simple example of epidemic containment. 
\end{abstract}

\section{\textbf{Introduction}}

In many real life situations, not only are we interested in understanding whether and how a system of interacting particles or individuals reach equilibria, but we also attempt to control or manage such equilibria so that the macroscopic behaviors of these systems reflect certain preferences. For example, how should the government control a flu outbreak by encouraging citizens to vaccinate? How should taxes be levied to influence people's  consumption, saving and investment decisions? How should an employer compensate its employees in order to boost  productivity? Present in all these scenarios are two parties: 1) the \emph{principal} who devises a \emph{contract}, according to which incentives are given to, or penalties are imposed on, 2) the \emph{agents}. We shall follow the directives of mainstream models and assume that the agents are \emph{rational} in the sense that they behave optimally to maximize their utilities, which translates the tradeoff between the rewards/penalties they received and the efforts they put in. The principal's problem is therefore to design a contract that maximizes the principal's own utility, which is often a function of the population's states and the transactions with the agents. To make the model even more realistic, we can consider a situation where the principal can only partially observe the agents' actions. We can also model the agents' choices of accepting or declining the contract by introducing a reservation utility.

\vspace{2mm}
Historically, principal / agent problems were first studied under the framework of economic contract theory, and most contributions deal with models involving a principal and a single agent. In the seminal work \cite{holmstrom1991multitask}, the authors considered a discrete-time model in which the agent's effort influences the drift of the state process. By assuming that both the principal and the agent have a CARA utility function, the authors showed that the optimal contract is linear. This model was further extended in \cite{schattler1993contract}, \cite{sung1995contract}, \cite{muller1998contract}, and \cite{hellwig2002contract}. The breakthrough in understanding the continuous-time principal agent problem can be attributed to Yuli Sannikov (see \cite{sannikov2008continuous} and \cite{sannikov2012contracts}) who exploited the dynamic nature of the agent's value function stemming from the dynamic programming principle. This remarkable observation allows the principal's optimal contracting problem to be formulated as a tractable optimal control problem. This approach was further investigated and generalized in \cite{cvitanic2015pa} and \cite{cvitanic2016moral} with the help of  second-order Backward Stochastic Differential Equations (BSDEs).

\vspace{2mm}
It is however a prevailing situation in real life applications that the principal faces multiple agents, and sometimes a large population of agents. Needless to say, the interactions and the competition between the agents are likely to play a key role in shaping their behaviors. This adds an extra layer of complexity since the principal not only needs to worry about the optimal responses of the individual agents, but also if an equilibrium resulting from the interactions among the agents is possible. Equilibria formed by a large population of agents with complex dynamics used to be notoriously intractable, not to mention how to solve an optimization problem on the equilibria which are parameterized by principal's contracts. To circumvent these difficulties, \cite{elie2016pa} borrowed the idea from the theory of mean field game and considered the limit scenario of infinitely many agents. Relying on the weak formulation of stochastic differential mean field games developed in \cite{carmona2015probabilistic}, the authors obtained a succinct representation of the equilibrium strategies of individual agents as well as the evolution of the equilibrium distribution across the population. Based on this convenient description of the equilibria, the author managed to obtain a tractable formulation of the optimal contracting problem, which can be solved by the technique of McKean-Vlasov optimal control.

\vspace{2mm}
Harnessing the probabilistic approach to finite state mean field games developed in \cite{carmona2018discrete}, we believe that the approach proposed in \cite{elie2016pa} can be extended to principal agent problem in finite state spaces. In this paper, we consider a game involving one principal and a population of infinitely many agents whose states belong to a finite space. Leveraging the mean field game models proposed and analyzed in \cite{carmona2018discrete}, we use continuous-time Markov chains to model the evolutions of the agents' states, and we assume that the transition rates between different states are functions of the agents' controls and the statistical distributions of all the agents' states. At the beginning of the game, the principal chooses a continuous reward rate and a terminal reward to be paid to each single agent. They are assumed to be functions of the past history of the agents' states. Each agent then maximizes their objective function by choosing their optimal effort, accounting for the principal's promised  reward and the states of all the other agents. Finally, the principal records a utility which depends upon the total reward offered to the agents as well as the states of the population of the agents. Assuming that the population of agents reaches a Nash equilibrium, the problem of the principal is to optimally choose a contract which will induce an equilibrium among the agents which achieves the maximal payoff.

\vspace{2mm}
Relying on the weak formulation developed in \cite{carmona2018discrete}, the Nash equilibrium can be readily characterized by a McKean-Vlasov BSDE. This BSDE is parameterized by the principal's continuous payment stream, as well as the principal's terminal payment (as terminal condition). Using ideas from \cite{cvitanic2016moral} and \cite{elie2016pa}, we show that controlling the terminal condition of the BSDE is indeed equivalent to controlling the initial condition and the martingale term of the corresponding Stochastic Differential Equation (SDE). This allows us to transform the rather intractable formulation of the principal's optimal contracting problem into an optimal control problem of a McKean-Vlasov SDE.

\vspace{2mm}
We then focus on a special case where the agents are risk-neutral in the terminal reward, the transition rates among their states are linear, and their instantaneous costs are quadratic in the control variable. In this case, we show that the optimal contracting problem can be further simplified to a deterministic control problem on the flow of distribution of the agents' states, and the agents use Markovian strategies when the principal announces the optimal contract. By applying the Pontryagin maximum principal, the optimal strategies of the agents and the resulting dynamics of the states under the optimal contract can be obtained by solving a system of forward-backward Ordinary Differential Equations (ODEs).

\vspace{2mm}
The rest of the paper is organized as follow. We introduce the model in Section \ref{sec:model}, where we  revisit some key elements of the weak formulation of finite state mean field games introduced and analyzed in \cite{carmona2018discrete}. We also formulate the principal's optimal contracting problem. Our first main result is stated in Section \ref{sec:mkv_ctrl}, where we show the equivalence between the optimal contracting problem and a McKean-Vlasov control problem. In Section \ref{sec:lq_model}, we solve explicitly the linear-quadratic model, where the search for the optimal contract can be reduced to a deterministic control problem on the flow of distribution of agents' states. Finally, we illustrate the above results in this linear-quadratic model with numerical computations performed on an example of epidemic containment  in Section \ref{sec:epidemic}.

\section{\textbf{Model Setup}}
\label{sec:model}

\subsection{Notations}
Throughout the paper, whenever $M$ is a square real matrix, we denote by  by $M^*$ its transpose and by $M^+$ its Moore-Penrose pseudo inverse. For a column vector $X$, we denote by $diag(X)$ the square diagonal matrix of which the diagonal elements are given by $X$. The multiplication is understood as matrix multiplication.

\subsection{Agent's State Process}\label{subsec:model_state_process}
We consider a game of duration $T$ involving a principal and a population of infinitely many agents. We assume that at any time $0\le t\le T$ each agent finds itself in one of $m$ different states. We denote by $E:=\{e_1, e_2, \dots, e_m\}$ the space of these possibles states, where the $e_i$'s are the basis vectors in $\mathbb{R}^m$. We denote by $\mathcal{S}$ the $m$-dimensional simplex $\mathcal{S} := \{ p \in \mathbb{R}^{m};\, \sum p_i = 1, p_i \ge 0\}$, which we identify with the space of probability measures on $E$.

\vspace{3mm}
Let $\Omega$ be the space of c\`adl\`ag (right continuous with left limits) functions from $[0,T]$ to $E$ which are left-continuous on $T$. Let $\bX=(X_t)_{0\le t\le T}$ be the canonical process. We interpret $X_t$ as a representative agent's state at time $t$. We denote by $\mathbb{F} = (\mathcal{F}_t)_{t\le T}$ with $\mathcal{F}_t := \sigma(\{X_s, s\le t\})$ the natural filtration generated by $\bX$ and $\mathcal{F} := \mathcal{F}_T$. Next, we fix an initial distribution $p^{\circ} \in \mathcal{S}$ for the agents' states. On $(\Omega, \mathbb{F}, \bar{\mathcal{F}})$ we consider the probability measure $\mathbb{P}$ under which the law of $X_0$ is $p^{\circ}$ and the canonical process $\bX$ is a continuous-time Markov chain where the transition rate between any two different states is $1$. We refer to $\mathbb{P}$ as the \emph{reference measure}. We recall that the process $\bX$ has the following semi-martingale representation (See \cite{elliott1995}, \cite{cohen2008}, \cite{cohen2010}):

\begin{equation}
\label{eq:canonical_representation}
X_t = X_0 + \int_0^t Q^0 X_{s-} ds + \mathcal{M}_t,
\end{equation}
where $\bcM=(\cM_t)_{0\le t\le T}$ is a $\mathbb{R}^m$-valued $\mathbb{P}$-martingale, and $Q^0$ is the square matrix with diagonal elements all equal to $-(m-1)$ and off-diagonal elements all equal to $1$. Moreover, the predictable quadratic variation of the martingale $\bcM$ under $\mathbb{P}$ is $\langle \bcM, \bcM \rangle_t = \int_0^t \psi_t dt$, where the matrix $\psi_t$ is defined as 
$$
\psi_t := diag(Q^0 X_{t-}) - Q^0 diag(X_{t-}) - diag(X_{t-}) Q^0.
$$ 
It is easy to verify that $\psi_t$ is a semi-definite positive matrix, and we define the corresponding (stochastic) seminorm $\|\cdot\|_{X_{t-}}$ on $\mathbb{R}^m$ by:
\begin{equation}\label{eq_ch2:seminorm}
\|z\|^2_{X_{t-}} := z^*\psi_t z.
\end{equation}
The seminorm $\|\cdot\|_{X_{t-}}$ can be written in a more explicit way. For $i\in\{1,\dots,m\}$, let us define the seminorm $\|\cdot\|_{e_i}$ on $\mathbb{R}^m$ by :

\begin{equation}
\label{eq_ch2:seminorm2}
\|z\|^2_{e_i} := z^*  [diag(Q^0 e_i) - Q^0 diag(e_i) - diag(e_i) Q^0] z = \sum_{j\neq i} |z_j - z_i|^2.
\end{equation}
Then it is easy to see that $\|z\|_{X_{t-}} = \sum_{i=1}^m \mathbbm{1}(X_{t-} = e_i)\|z\|_{e_i}$.

\vspace{3mm}
Let $A$ be a compact subset of $\mathbb{R}^l$ in which the agents can choose their controls. Denote by $\mathbb{A}$ the collection of $\mathbb{F}$-predictable processes $\balpha=(\alpha_t)_{0\le t\le T}$ taking values in $A$. We introduce the transition rate function $q: [0,T] \times \{1,\dots,m\}^2 \times A \times \mathcal{S} \ni (t,i,j,\alpha,p)\rightarrow q(t, i, j, \alpha, p)\in\mathbb{R}$, and we denote by $Q(t,\alpha,p)$ the transition rate matrix $[q(t, i, j, \alpha, p)]_{1\le i,j\le m}$. In the rest of the paper, we make the following assumption on the transition rate function $q$:

\begin{hypothesis}\label{hypo:transrateminor}
(i) For all $(t,i,\alpha, p) \in [0,T] \times \{1,\dots,m\}\times A \times \mathcal{S}$, we have $\sum_{j=1}^m q(t, i, j, \alpha, p) = 0$.

\noindent(ii) There exist constants $C>0$ such that for all $(t,i,j,\alpha,p) \in [0,T]\times \{1,\dots,m\}^2\times A\times\mathcal{S}$ with $i\neq j$, we have $0 < q(t,i,j,\alpha,p) \le C$.

\noindent(iii) There exists a constant $C>0$ such that for all $(t,i,j)\in[0,T]\times \{1,\dots,m\}^2$, $\alpha,\alpha' \in A$, $p,p' \in \mathcal{S}$, we have:
{\small
\begin{equation}
|q(t,i,j,\alpha,p) - q(t,i,j,\alpha', p')| \le C(\|\alpha - \alpha'\| + \|p - p'\|).
\end{equation}
}
\end{hypothesis}

We  use the weak formulation to specify how the states of the agents evolve. This means that we specify how the agents' controls and the states of the other agents determine the probability distribution of any given agent's state via its stochastic transition rate, rather than the dynamics of the state process sample path. In addition, we assume that the interactions between the agents are \emph{ mean field}, in the sense that each agent feels the impact of the other agents via the statistical distribution of their states. Let us fix $\bp=(p(t))_{0\le t\le T}$ a flow of probability measures in $E$ such that $p(0) = p^{\circ}$. $p(t)$ represents the statistical distribution of the agents' states at time $t$. We also fix a control $\balpha\in\mathbb{A}$. By Girsanov's theorem (see Theorem III.41 in \cite{protter2005stochastic} or Lemma 4.3 in \cite{sokol2015exponential}), we construct a probability measure $\mathbb{Q}^{(\balpha,\bp)}$ which is equivalent to $\mathbb{P}$ and such that $\bX$ admits the following representation:

\begin{equation}
\label{eq_ch3:X_decomp}
X_t = X_0 + \int_{(0,t]} Q^*(s, \alpha_s, p(s)) X_{s-} dt + \mathcal{M}^{(\balpha,\bp)}_t,
\end{equation}
where $\mathcal{M}^{(\balpha,\bp)}_t$ is a martingale under the measure $\mathbb{Q}^{(\balpha,\bp)}$. The Radon-Nikodym derivative of $\mathbb{Q}^{(\balpha,\bp)}$ with regard to $\mathbb{P}$ is given by $\frac{d\mathbb{Q}^{(\balpha, \bp)}}{d\mathbb{P}} = \mathcal{E}(\bL^{(\balpha,\bp)})_T$, where $\mathcal{E}(\bL^{(\balpha,\bp)})$ is the Dol\'eans-Dade exponential of the martingale $\bL^{(\balpha,\bp)}$ defined as follows:
\begin{equation}\label{eq_ch3:martingale_girsanov}
dL^{(\balpha,\bp)}_t = X_{t-}^* (Q(t, \alpha_t, p(t)) - Q^0)\psi_t^+ d\mathcal{M}_t,\;\; L^{(\balpha,\bp)}_0 = 0.
\end{equation}
The representation of $\bX$ in equation (\ref{eq_ch3:X_decomp}) means that under the measure $\mathbb{Q}^{(\balpha,\bp)}$, the stochastic intensity rate of jumps for the state process $\bX$ is given by the matrix $Q(t, \alpha_t, p(t))$. In addition, since $\mathbb{Q}^{(\balpha,\bp)}$ and the reference measure $\mathbb{P}$ coincide on $\mathcal{F}_0$, the distributions of $X_0$ under $\mathbb{Q}^{(\balpha,\bp)}$ and under $\mathbb{P}$ are both $p^{\circ}$. In particular, when $\balpha$ is a Markov control, i.e. of the form $\alpha_t = \phi(t, X_{t-})$ for some measurable feedback function $\phi$, the canonical process $\bX$ becomes a continuous-time Markov chain under the measure $\mathbb{Q}^{(\balpha,\bp)}$ for which the transition rate from state $i$ to state $j$ at time $t$ is $q(t,i,j,\phi(t,i),p(t))$. See \cite{carmona2018discrete} for more details about the construction of the probability measure $\mathbb{Q}^{(\balpha,\bp)}$.

\subsection{Principal's Contract and Agent's Objective Function}
\label{subsec:model_objective}

The principal chooses a contract consisting of a payment stream $\br=(r_t)_{0\le t\le T}$ which is a $\mathbb{F}-$predictable process and a terminal payment $\xi$ which is a $\mathcal{F}_T-$ measurable $\mathbb{P}-$square integrable random variable. Let $c:[0,T] \times E \times A \times \mathcal{S} \rightarrow \mathbb{R}$ be the running cost function of the typical agent. Also, let $u: \mathbb{R} \rightarrow \mathbb{R}$ (resp. $U: \mathbb{R} \rightarrow \mathbb{R}$) be the agent's utility function with regard to the continuous payments (resp. the terminal payment). If an agent accepts the contract $(\br,\xi)$ and the statistical distribution of his/her state at time $t$ is denoted $p(t)$ for $0\le t\le T$, the agent's expected total cost, $J^{\br,\xi}(\balpha,\bp)$, is defined as:
\begin{equation}\label{eq:total_cost_agent}
J^{\br,\xi}(\balpha,\bp):=\mathbb{E}^{\mathbb{Q}^{(\balpha,\bp)}}\left[\int_0^T [c(t, X_t, \alpha_t, p(t)) - u(r_t)] dt - U(\xi)\right].
\end{equation}

\subsection{Principal's Objective Function}
\label{subsec:model_objective_principal}

The principal's objective function depends on the distribution of the agents' states and the payments made to the agents. Let $c_0:[0,T]\times\mathcal{S} \rightarrow \mathbb{R}$ be the running cost function and $C_0:\mathcal{S} \rightarrow \mathbb{R}$ the terminal cost function resulting from the distribution of the agents' behaviors. Given that all the agents choose $\balpha$ as their control strategy,  the distribution of the agents' states is given by $\bp=(p(t))_{t\in[0, T]}$, and the  contract offered by the principal is $(\br, \xi)$,  the principal records the expected total cost, $J_0^{\balpha,\bp}(\br, \xi)$, defined as:
\begin{equation}\label{eq:total_cost_principal}
J_0^{\balpha,\bp}(\br,\xi):=\mathbb{E}^{\mathbb{Q}^{(\balpha,\bp)}}\left[\int_0^T [c_0(t, p(t)) + r_t] dt +C_0(p(T)) + \xi\right].
\end{equation}

\subsection{Agents' Mean Field Nash Equilibria}
\label{subsec:model_agent_eq}

We assume that the population of infinitely many agents reach a Nash equilibrium for a given contract $(\br,\xi)$ proposed by the principal. We recall the following definition of Nash equilibria in the \emph{weak formulation}  introduced in \cite{carmona2018discrete}, and adapted to the present situation. See Definition 2.8 therein.

\begin{definition}
\label{def:equilibrium}
Let $(\br,\xi)$ be a contract, $\hat \bp:[0,T]\rightarrow \mathcal{S}$ be a measurable mapping such that $\hat p(0) = p^{\circ}$, and $\hat \balpha\in\mathbb{A}$ be an admissible control for the agents. We say that the couple $(\hat \balpha,\hat \bp)$ is a Nash equilibrium for the contract $(\br,\xi)$, in which case we use the notation $(\hat \balpha,\hat \bp) \in \mathcal{N}(\br,\xi)$, if:

\vspace{2mm}
\noindent (i) $\hat\balpha$ minimizes the cost when the agent is committed to the contract $(\br,\xi)$ and the distribution of all the agents is given by the flow $\hat \bp$, i.e. if:
\begin{equation}\label{eq:def_nasheq_optimality}
\hat\balpha = \arg\inf_{\balpha \in \mathbb{A}}\mathbb{E}^{\mathbb{Q}^{(\balpha,\hat \bp)}}\left[\int_0^T [c(t, X_t, \alpha_t, \hat p(t)) - u(r_t)] dt - U(\xi)\right].
\end{equation}
\noindent (ii) $(\hat\balpha,\hat \bp)$ satisfies the consistency conditions:
\begin{equation}\label{eq:def_nasheq_cons1}
\forall t \in [0,T] \qquad \hat p(t) = \mathbb{E}^{\mathbb{Q}^{(\hat\balpha,\hat \bp)}}[X_t].
\end{equation}
Note that equation (\ref{eq:def_nasheq_cons1}) is equivalent to $\hat p_i(t) = \mathbb{Q}^{(\hat\balpha,\hat \bp)}[X_t = e_i]$, for all $t \in [0,T]$ and  $i\in\{1,\dots,m\}$.
\end{definition}

\subsection{Principal's Optimal Contracting Problem}
\label{subsec:model_optimal_contract}

We now turn to the principal's optimal choice of the contract. This amounts to minimizing the objective function computed based on the Nash equilibria formed by the agents. Of course we need to address the existence of Nash equilibria. However, in the following formulation of the optimal contracting problem, we shall avoid the problem of existence by only considering the contracts $(\br,\xi)$ that result in at least one Nash equilibrium. We call such contracts \emph{admissible contracts}, and we denote by $\mathcal{C}$ the collection of all admissible contracts. In addition, among all the possible Nash equilibria, we would like to disregard the equilibria in which the agent's expected total cost is above a given threshold $\kappa$. The motivation for imposing this additional constraint is to model the \emph{take-it-or-leave-it} behavior of the agents in contract theory: if the agent's expected total cost exceeds a certain threshold, it should be able to turn down the contract. Summarizing the constraints mentioned above, we propose the following optimization problem for the principal:
{\small
\begin{equation}\label{eq:principal_pb}
V(\kappa) := \inf_{(\br,\xi)\in\mathcal{C}}\inf_{\substack{(\balpha,\bp) \in \mathcal{N}(\br,\xi)\\ J^{\br,\xi}(\balpha,\bp) \le \kappa}}\mathbb{E}^{\mathbb{Q}^{(\balpha,\bp)}}\left[\int_0^T [c_0(t, p(t)) + r_t] dt +C_0(p(T)) +\xi\right],
\end{equation}
}
where we adopt the convention that the infimum over an empty set equals $+\infty$. 

\section{\textbf{A McKean-Vlasov Control Problem}}
\label{sec:mkv_ctrl}

The original formulation of the principal's optimal contract is intuitive but far from tractable. In this section we  provide an equivalent formulation which turns out to be an optimal control problem of McKean-Vlasov type. To this end, we rely on the probabilistic characterization of the agents' Nash equilibria developed in \cite{carmona2018discrete}.

\subsection{Agent's Optimization Problem}
\label{subsec:optimization_agent}

We define the Hamiltonian  $H: [0,T] \times E \times \mathbb{R}^m\times A \times \mathcal{S} \times \mathbb{R} \rightarrow \mathbb{R} $ for the agent's optimization problem by:

\begin{equation}
\label{eq:hamiltonian_def_agent}
H(t,x,z,\alpha,p, r) := c(t,x,\alpha,p) - r + x^*(Q(t, \alpha, p) - Q^0)z.
\end{equation}
Recall that the canonical process $\bX$ representing the typical agent's state takes value in the set $E = \{e_1,\dots,e_m\}$. We define the reduced Hamiltonian by $H_i(t,z,\alpha,p,R):=H(t,e_i,z,\alpha,p,R)$ for $i=1,dots,m$. It is straightforward to check:
\begin{equation}\label{eq:reduced_hamiltonian_def}
H_i(t,z,\alpha,p, r) := c(t,e_i, \alpha,p) - r + \sum_{j\neq i} (z_j - z_i)(q(t, i, j, \alpha, p) - 1).
\end{equation} 
We make the following assumption on the Hamiltonian minimizarion. Clearly, minimizers of $H_i$ should not depend on $r$.

\begin{hypothesis}
\label{hypo:lipschitz_optimizer}
For all $0\le t\le T$, $i\in\{1,\dots,m\}$, $z\in\mathbb{R}^m$ and $p\in\mathcal{S}$, the mapping $\alpha \rightarrow H_i(t,z,\alpha,p,r)$ admits a unique minimizer. We denote the minimizer by $\hat a_i(t,z,p)$. In addition, for all $i\in\{1,\dots,m\}$, we assume that $\hat a_i$ is a measurable mapping on $[0,T]\times\mathbb{R}^m\times\mathcal{S}$, and there exists a constant $C>0$ such that for all $i\in\{1,\dots,m\}$, $z, z'\in\mathbb{R}^m$ and $p\in\mathcal{S}$:
\begin{equation}\label{eq:optimizer_lipschitz}
\|\hat a_i(t,z,p) - \hat a_i(t,z',p)\| \le C\|z - z'\|_{e_i}
\end{equation}
\end{hypothesis}

\begin{remark}
Assumption \ref{hypo:lipschitz_optimizer} holds, for example, when the cost function $c$ is strongly convex in $\alpha$ and the transition rate function is linear in $\alpha$. This can be proved by following the arguments in the proof of Lemma 3.3, Chapter 3 in \cite{carmona2017book}, although the later is stated in the context of the stochastic optimal control of SDEs, which has a slightly different form of Hamiltonian.
\end{remark}

We denote by $\hat H_i$ the minimum of the reduced Hamiltonian: 
\begin{equation}
\hat H_i(t,z,p,r) := \hat H_i(t,z,\hat a_i(t,z,p),p,r). 
\end{equation}
Now we define the mappings $\hat H$ and $\hat a$ by:
\begin{align}
\hat H(t,x,z,p,r) :=&\;\; \sum_{i=1}^m \hat H_i(t,z,p,r) \mathbbm{1}(x = e_i),\label{eq:def_h_hat}\\
\hat a(t,x,z,p) :=&\;\; \sum_{i=1}^m \hat a_i(t,z,p) \mathbbm{1}(x = e_i).\label{eq:def_a_hat}
\end{align}
Under Assumption \ref{hypo:lipschitz_optimizer}, it is clear that $\hat a(t,x,z,p)$ is the unique minimizer of the mapping $\alpha \rightarrow H(t,x,z,\alpha,p,r)$, and the minimum is given by $\hat H(t,x,z,p,r)$. In addition, from Assumption \ref{hypo:transrateminor}, we can show the following result on the Lipschitz property of $\hat H$:

\begin{lemma}
\label{lem:lip_h_hat}
There exists a constant $C>0$ such that for all $(\omega,t)\in\Omega\times(0,T]$, $r\in\mathbb{R}$, $p\in\mathcal{S}$ and $z,z'\in\mathbb{R}^m$, we have:
\begin{equation}\label{eq:lipschitz_h_hat}
|\hat H(t,X_{t-},z,p,r) - \hat H(t,X_{t-},z',p,r)| \le C\|z - z'\|_{X_{t-}}
\end{equation}
\begin{equation}\label{eq:lipschitz_full_optimizer}
|\hat a(t,X_{t-},z,p) - \hat a(t,X_{t-},z',p)| \le C\|z - z'\|_{X_{t-}}.
\end{equation}
\end{lemma}

The total expected cost of the agents and the value function of the agent's optimization problem can be characterized by BSDEs driven by continuous-time Markov chain. We refer the readers to \cite{cohen2008} and \cite{cohen2010} for the general theory for this type of BSDE. Let us fix a contract $(\br,\xi)$ and a measurable function $\bp:[0,T]\rightarrow \mathcal{S}$. Given $\balpha\in\mathbb{A}$, we consider the following BSDEs:
\begin{equation}\label{eq:bsde_agent_total_cost}
Y_t = -U(\xi) + \int_t^T H(s, X_{s-}, Z_s, \alpha_s, p(s), u(r_t)) ds - \int_t^T Z_s^*d\mathcal{M}_s.
\end{equation}
\begin{equation}\label{eq:bsde_agent_optimization}
Y_t = -U(\xi) + \int_t^T \hat H(s, X_{s-}, Z_s, p(s), u(r_t)) ds - \int_t^T Z_s^*d\mathcal{M}_s.
\end{equation}
In \cite{carmona2018discrete}, we proved the following result:

\begin{lemma}
\label{lem:bsde_characterization}
For each fixed contract $(\br,\xi)$, $\balpha\in\mathbb{A}$ and measurable mapping $\bp: [0,T]\rightarrow\mathcal{S}$,

\vspace{1mm}
\noindent(i) the BSDE (\ref{eq:bsde_agent_total_cost}) admits a unique solution $(\bY, \bZ)$ and we have $J^{\br,\xi}(\balpha,\bp) = \mathbb{E}^{\mathbb{P}}[Y_0]$.

\vspace{1mm}
\noindent(ii) The BSDE (\ref{eq:bsde_agent_optimization}) admits a unique solution $(\bY,\bZ)$ and we have $\inf_{\balpha \in \mathbb{A}}J^{\br,\xi}(\balpha,\bp) = \mathbb{E}^{\mathbb{P}}[Y_0]$. In addition, the optimal control of the agent is $\hat a(t, X_{t-}, Z_t, p(t))$.
\end{lemma}

\subsection{Representation of Nash Equilibria Based on BSDEs}
\label{subsec:mkv_ctrl_nash_eq}
Now we consider the following McKean-Vlasov BSDE:
\begin{align}
Y_t =&\;\; -U(\xi) + \int_t^T \hat H(s, X_{s-}, Z_s, p(s), u(r_s)) ds - \int_t^T Z_s^* d\mathcal{M}_s,\label{eq:bsde_mfg_1}\\
\mathcal{E}_t =&\;\;  1 + \int_0^t \mathcal{E}_{s-} X_{s-}^* (Q(s, \alpha_s, p(s)) - Q^0)\psi_s^+ d\mathcal{M}_s,\label{eq:bsde_mfg_2}\\
\alpha_t = &\;\; \hat a(t,X_{t-},Z_t,p(t)),\label{eq:bsde_mfg_3}\\
p(t) = &\;\; \mathbb{E}^{\mathbb{Q}}[X_t],\;\;\frac{d\mathbb{Q}}{d\mathbb{P}} = \mathcal{E}_T.
\label{eq:bsde_mfg_4}
\end{align}

\begin{definition}
\label{def:mkv_bsde_sol}
We say that a tuple $(\bY,\bZ,\balpha,\bp,\mathbb{Q})$ is a solution to the McKean-Vlasov BSDE (\ref{eq:bsde_mfg_1})-(\ref{eq:bsde_mfg_4}) if $\bY$ is an adapted  c\`adl\`ag process such that $\mathbb{E}^{\mathbb{P}}[\int_0^T Y_t^2] < +\infty$ for all $t \in [0,T]$, $\bZ$ is an adapted left-continuous process such that $\mathbb{E}^{\mathbb{P}}[\int_0^T \|Z_t\|_{X_{t-}}^2] < +\infty$, $\balpha\in\mathbb{A}$, $\bp:[0,T]\rightarrow\mathcal{S}$ is a measurable mapping, $\mathbb{Q}$ is a probability measure on $\Omega$ and equations (\ref{eq:bsde_mfg_1})-(\ref{eq:bsde_mfg_4}) are satisfied $\mathbb{P}$-a.s. for all $0\le t\le T$.
\end{definition}

The following result links the solution of the McKean-Vlasov BSDE (\ref{eq:bsde_mfg_1})-(\ref{eq:bsde_mfg_4}) to the Nash equilibria of the agents.

\begin{theorem}
\label{thm:proba_characterization_ne}
Let Assumption \ref{hypo:transrateminor} and Assumption \ref{hypo:lipschitz_optimizer} hold, and let $(\br,\xi)$ be a contract. If the BSDE (\ref{eq:bsde_mfg_1})-(\ref{eq:bsde_mfg_4}) admits a solution $(\bY,\bZ,\balpha,\bp,\mathbb{Q})$, then $(\balpha,\bp)$ is a Nash equilibrium. Conversely if $(\hat\balpha,\hat \bp)$ is a Nash equilibrium, then the BSDE (\ref{eq:bsde_mfg_1})-(\ref{eq:bsde_mfg_4}) admits a solution $(\bY,\bZ,\balpha,\bp,\mathbb{Q})$ such that $\balpha = \hat\balpha$, $d\mathbb{P}\otimes dt$-a.e. and $p(t) = \hat p(t)$ $dt$-a.e.
\end{theorem}

\begin{proof}
Let $(\bY,\bZ,\balpha,\bp,\mathbb{Q})$ be a solution to the system (\ref{eq:bsde_mfg_1})-(\ref{eq:bsde_mfg_4}). We use Lemma \ref{lem:bsde_characterization} to check that $\balpha$ is the optimal control of the agent when all the agents are committed to the contract $(\br,\xi)$ and the distribution of their states is given by $\bp$. Indeed, from equations (\ref{eq:bsde_mfg_2}) and (\ref{eq:bsde_mfg_4}), we see that item (ii) in Definition \ref{def:equilibrium} is verified. Therefore $(\balpha,\bp)$ is a Nash equilibrium.

\vspace{2mm}
We now show the second part of the claim. Let $(\hat\balpha,\hat \bp)$ be a Nash equilibrium and set $\hat{\mathbb{Q}} := \mathbb{Q}^{(\hat\balpha,\hat \bp)}$.  By item (ii) in Definition \ref{def:equilibrium}, we see that (\ref{eq:bsde_mfg_2}) and (\ref{eq:bsde_mfg_4}) are satisfied. Therefore it only remains to show (\ref{eq:bsde_mfg_1}) and (\ref{eq:bsde_mfg_3}). Let us consider the BSDE:
\begin{equation}\label{eq:proof_proba_characterization_bsde1}
Y_t = -U(\xi) + \int_t^T H(s, X_{s-}, Z_s, \hat\alpha_s, \hat p(s), u(r_s)) ds - \int_t^T Z_s^*\cdot d\mathcal{M}_s.
\end{equation}
By Lemma \ref{lem:bsde_characterization}, the above BSDE admits a unique solution which we denote by $(\bY^0, \bZ^0)$. In addition, we have $\mathbb{E}^{\mathbb{P}}[Y^0_0] = J^{\br,\xi}(\hat\balpha,\hat \bp)$. On the other hand, we consider the following BSDE:
\begin{equation}\label{eq:proof_proba_characterization_bsde2}
Y_t = -U(\xi) + \int_t^T \hat H(s, X_{s-}, Z_s, \hat p(s), u(r_s)) ds - \int_t^T Z_s^*\cdot d\mathcal{M}_s.
\end{equation}
From Lemma \ref{lem:bsde_characterization}, we see that BSDE (\ref{eq:proof_proba_characterization_bsde2}) also admits a unique solution which we denote by $(\bY^1, \bZ^1)$, and we have $\mathbb{E}^{\mathbb{P}}[Y^1_0]  = \inf_{\balpha\in\mathbb{A}}J^{\br,\xi}(\balpha,\hat \bp)$. By the definition of the Nash equilibrium, we have:
\[
\hat\balpha \in \arg\inf_{\balpha\in\mathbb{A}} J^{\br,\xi}(\balpha,\hat p),
\]
which implies that $\mathbb{E}^{\mathbb{P}}[Y_0^1] = \mathbb{E}^{\mathbb{P}}[Y_0^0]$. Note that $Y_0^1$ and $Y_0^0$ are $\mathcal{F}_0$-measurable and $\mathbb{P}$ coincides with $\hat{\mathbb{Q}}$ on $\mathcal{F}_0$. Therefore we have $\mathbb{E}^{\hat{\mathbb{Q}}}[Y_0^1 - Y_0^0] = 0$. Now we set $\hat\alpha'_t:=\hat a(t,X_{t-}, Z_t^1, \hat p(t))$ which minimizes the mapping $\alpha \rightarrow H(s, X_{t-}, Z^1_t, \alpha, \hat p(t), u(r_t))$. Since $(\bY^1, \bZ^1)$ solves the BSDE (\ref{eq:proof_proba_characterization_bsde2}), we have:
\begin{equation}\label{eq:proof_proba_characterization_bsde3}
Y^1_t = -U(\xi) + \int_t^T H(s, X_{s-}, Z^1_s, \hat\alpha'_t, \hat p(s), u(r_s)) ds - \int_t^T (Z^1_s)^*\cdot d\mathcal{M}_s.
\end{equation}
Taking the difference of the BSDEs (\ref{eq:proof_proba_characterization_bsde3}) and (\ref{eq:proof_proba_characterization_bsde1}), we obtain:
\begin{align*}
Y_0^0 - Y_0^1&=\;\;\int_0^T \big[H(t, X_{t-}, Z^0_t, \hat\alpha_t, \hat p(t), u(r_t)) - H(t, X_{t-}, Z^1_t, \hat \alpha'_t, \hat p(t), u(r_t))\big]dt
- \int_0^T (Z^0_t - Z^1_t)^*\cdot d\mathcal{M}_t\\
&=\int_0^T \big[c(t, X_{t-}, \hat\alpha_t, \hat p(t)) - c(t, X_{t-}, \hat\alpha'_t, \hat p(t))\big] dt\\
&\;\; +\int_0^T\big[ X_{t-}^*(Q(t,\hat\alpha_t, \hat p(t)) - Q^0) Z^0_t - X_{t-}^*(Q(t,\hat\alpha'_t, \hat p(t)) - Q^0) Z^1_t\big]dt - \int_0^T (Z^0_t - Z^1_t)^*\cdot d\mathcal{M}_t\\
&=\int_0^T \big[c(t, X_{t-}, \hat\alpha_t, \hat p(t)) - c(t, X_{t-}, \hat\alpha'_t, \hat p(t)) + X_{t-}^*(Q(t,\hat \alpha_t, \hat p(t)) - Q(t,\hat\alpha'_t, \hat p(t)))Z_t^1\big]dt\\
&\hskip 75pt
-  \int_0^T (Z^0_t - Z^1_t)^*\cdot d\mathcal{M}^{(\hat\alpha, \hat p)}_t.
\end{align*}
Using the optimality of $\hat\alpha'_t$ and taking the expectation under the measure $\mathbb{Q}^{(\hat\balpha, \hat \bp)}$, we obtain:
\begin{equation}\label{eq:proof_th_mkv_bsde}
0 = \mathbb{E}^{\mathbb{Q}^{(\hat\balpha, \hat \bp)}}[H(t, X_{t-}, Z^1_t, \hat\alpha_t, \hat p(t), u(r_t))-H(t, X_{t-}, Z^1_t, \hat \alpha'_t, \hat p(t), u(r_t))] \ge 0.
\end{equation}
Assume that there exists a measurable subset $N$ of $[0,T]\times\Omega$ with strictly positve $d\mathbb{Q}^{(\hat\balpha, \hat \bp)}\otimes dt$ measure, such that $\hat\alpha'_t \neq \hat\alpha_t$ for $(\omega, t)\in N$. By Assumption \ref{hypo:lipschitz_optimizer}, the mapping $\alpha \rightarrow H(t, X_{t-}, Z_t^{1}, \alpha, \hat p(t), u(r_t))$ admits a unique minimizer and therefore for all $(\omega,t) \in N$, we have:
\[
H(t, X_{t-}, Z^1_t, \hat\alpha_t, \hat p(t), u(r_t)) > H(t, X_{t-}, Z^1_t, \hat \alpha'_t, \hat p(t), u(r_t)).
\]
Piggybacking on the argument laid out above, we see that the second inequality is strict in (\ref{eq:proof_th_mkv_bsde}) which leads to a contradiction.Therefore we have $\hat\alpha'_t = \hat\alpha_t$, $d\mathbb{Q}^{(\hat\balpha, \hat \bp)}\otimes dt$-a.e. Since $\mathbb{Q}^{(\hat\balpha, \hat \bp)}$ and $\mathbb{P}$ are equivalent, we have $\hat\alpha'_t = \hat\alpha_t$, $d\mathbb{P}\otimes dt$-a.e.

\vspace{2mm}
Comparing BSDEs (\ref{eq:proof_proba_characterization_bsde1}) and (\ref{eq:proof_proba_characterization_bsde3}) and using the uniqueness of the solution, we get $\mathbb{E}[\int_0^T\|Z_t^1 - Z_t^0\|^2_{X_{t-}}] = 0$. Now by the regularity of $\hat a$ in Assumption \ref{hypo:lipschitz_optimizer}, we have:
\begin{align*}
\mathbb{E}\left[\int_0^T\|\hat \alpha_t - \hat a(t, X_{t-}, Z_t^0, \hat p(t))\|^2 dt\right] & = \mathbb{E}\left[\int_0^T\|\hat \alpha'_t - \hat a(t, X_{t-}, Z_t^0, \hat p(t))\|^2 dt\right] \\
& =\mathbb{E}\left[\int_0^T\|\hat a(t,X_{t-}, Z_t^1, \hat p(t)) -  \hat a(t, X_{t-}, Z_t^0, \hat p(t))\|^2 dt\right]\\
& \le  C \mathbb{E}[\int_0^T\|Z_t^1 - Z_t^0\|^2_{X_{t-}}]\\
&  = 0.
\end{align*}
Therefore we have $\hat \alpha_t = \hat a(t, X_{t-}, Z_t^0, \hat p(t))$, $d\mathbb{P}\otimes dt$-a.e., which immediately implies (\ref{eq:bsde_mfg_1}) and (\ref{eq:bsde_mfg_3}). This completes the proof.
\end{proof}

\subsection{Principal's Optimal Contracting Problem}
\label{subsec:mkv_ctrl_optimal_contract}

We now turn to the principal's optimal choice of the contract. Recall that we have defined $\mathcal{C}$ to be the collection of all contracts that result in at least one Nash equilibrium. In addition, in order to model the fact that agents are not accepting contracts that will incur a cost above a certain (reservation) threshold, we further restrict our optimization to the collection of Nash equilibria in which the agent's expected total cost is below a given threshold $\kappa$. Therefore we propose to solve the following optimization problem for the principal:
\[
V(\kappa) := \inf_{(\br,\xi)\in\mathcal{C}}\inf_{\substack{(\balpha,\bp) \in \mathcal{N}(\br,\xi)\\ J^{\br,\xi}(\balpha,\bp) \le \kappa}}\mathbb{E}^{\mathbb{Q}^{(\balpha,\bp)}}\left[\int_0^T [c_0(t, p(t)) + r_t] dt +C_0(p(T)) +\xi\right],
\]
where we adopt the convention that the infimum over an empty set equals $+\infty$. 

\vspace{2mm}
Formulated in this way, the problem seems rather intractable. However, thanks to the probabilistic characterization of agents' Nash equilibria stated in Theorem \ref{thm:proba_characterization_ne}, it is possible to transform it into a McKean-Vlasov control problem. Let us denote by $\mathcal{H}_{X}^2$ the collection of $\mathbb{F}-$adapted and left-continuous processes $\bZ$ such that $Z_t \in \mathbb{R}^m$ for $0\le t\le T$ and $\mathbb{E}[\int_0^T \|Z_t\|_{X_{t-}} dt] < +\infty$. We also denote by $\mathcal{R}$ the collection of $\mathbb{F}-$predictable process taking values in $\mathbb{R}$. Given $\bZ\in\mathcal{H}_{X}^2$, $\br \in \mathcal{R}$ and $Y_0$ a $\mathcal{F}_0$-measurable random variable, we consider the following SDE of McKean-Vlasov type:
\begin{align}
Y_t =&\;\;Y_0-\int_0^t\hat H(s, X_{s-}, Z_s, p(s), u(r_s)) ds + \int_0^t Z_s^* d\mathcal{M}_s,\label{eq:forward_y_mkv1}\\
\mathcal{E}_t =&\;\;1 + \int_0^t \mathcal{E}_{s-} X_{s-}^* (Q(s, \alpha_s, p(s)) - Q^0)\psi_s^+ d\mathcal{M}_s,
\label{eq:forward_y_mkv2}\\
\alpha_t = &\;\;\hat a(t,X_{t-},Z_t,p(t)),\label{eq:forward_y_mkv3}\\
p(t) = &\;\;\mathbb{E}^{\mathbb{Q}}[X_t],\;\;\frac{d\mathbb{Q}}{d\mathbb{P}} = \mathcal{E}_T.
\label{eq:forward_y_mkv4}
\end{align}
These are exactly the same equations as in the McKean-Vlasov BSDE (\ref{eq:bsde_mfg_1})-(\ref{eq:bsde_mfg_4}), except that we write the dynamic of $\bY$ in the forward direction of time. Let us denote its solution by $(\bY^{\bZ,br, Y_0)},\bZ^{(\bZ,\br,Y_0)},\balpha^{(\bZ,\br,Y_0)}, \bp^{(\bZ,\br,Y_0)}, \mathbb{P}^{(\bZ,\br,Y_0)})$ and the expectation under $\mathbb{P}^{(\bZ,\br,Y_0)}$ by $\mathbb{E}^{(\bZ,\br,Y_0)}$ and let us consider the following optimal control problem:
\begin{equation}
\label{eq:principal_pb_new}
\tilde V(\kappa) := \inf_{\mathbb{E}^{\mathbb{P}}[Y_0] \le \kappa} \inf_{\substack{\bZ\in\mathcal{H}_{X}^2\\ \br \in \mathcal{R}}}\mathbb{E}^{(\bZ,\br,Y_0)}\bigg[\int_0^T [c_0(t, p^{(\bZ,\br,Y_0)}(t)) + r_t] dt +C_0(p^{(\bZ,\br,Y_0)}(T))+ U^{-1}(-Y_T^{(\bZ,\br,Y_0)})\bigg].
\end{equation}
As a direct consequence of Theorem \ref{thm:proba_characterization_ne}, we have the following result:
\begin{theorem}\label{thm:principal_pb}
Let Assumption \ref{hypo:transrateminor} and Assumption \ref{hypo:lipschitz_optimizer} hold. Then $\tilde V(\kappa) = V(\kappa)$.
\end{theorem}

\section{\textbf{Solving the Linear-Quadratic Model}}
\label{sec:lq_model}

Solving the contract theory problem in the full generality considered so far seems out of reach. So
in this section, we focus on a special setup of the principal agent problem where the transition rate has a linear structure and the cost function is quadratic in the control. When the agents are risk-neutral in term of the utility of their terminal reward, we show that the principal's optimal contracting problem can be further reduced to a deterministic control problem on the space of probability distributions of the agents' states.

\subsection{Model Setup}
\label{subsec:lq_model_setup}

We set the initial distribution of the agents to be $p^{\circ} \in \mathcal{S}$. Each agent is allowed to pick a control $\alpha$ which belongs to the bounded interval $A := [\ubar \alpha, \bar \alpha]\subset\mathbb{R}^+$. We assume that the transition rate is a linear function of the control and we define:
\begin{align}
q(t,i,j,\alpha,p) := \bar q_{i,j}(t,p) + \lambda_{i,j}(\alpha - \ubar\alpha),\;\;\text{for}\;\;i\neq j,\label{eq:q_matrix_lq}\\
q(t,i,i,\alpha,p) := -\sum_{j\neq i}q(t,i,j,\alpha,p),
\end{align}
where we assume that $\lambda_{i,j} \in \mathbb{R}^+$ for all $i\neq j$, $\sum_{j\neq i}\lambda_{i,j} > 0$ for all $i$, and $\bar q_{i,j}:[0,T]\times\mathcal{S}\rightarrow\mathbb{R}^+$ are continuous mappings for all $i\neq j$. We assume that the cost function of a typical agent (not including the utility derived from the payment stream) takes the following form:
\begin{equation}\label{eq:agent_cost_lq}
c(t, e_i, \alpha, p) := c_1(t, e_i, p) + \frac{\gamma_i}{2}\alpha^2,
\end{equation}
where $\gamma_i>0$, and the mapping $(t,p)\rightarrow c_1(t, e_i, p)$ is continuous for all $i\in\{1,\dots,m\}$. Finally we define the agent's utility function of terminal reward to be $U(\xi) = \xi$ and the utility function of continuous reward $u$ to be a continuous, concave and increasing function.

\vspace{2mm}
Under the setup outlined above, it is straightforward to compute the optimizer of the Hamiltonian $H$ defined in (\ref{eq:hamiltonian_def_agent}). We get:
\begin{equation}\label{eq:hamiltonian_optimizer_lq}
\hat a(t,e_i, z,p) = \hat a(e_i, z) = b\left(-\frac{1}{\gamma_i}\sum_{j\neq i}\lambda_{i,j}(z_j- z_i)\right),
\end{equation}
for $i\in \{1,\dots,m\}$, where we have defined $b(z) := \min\{\max\{z,\ubar\alpha\},\bar\alpha\}$.

\subsection{Reduction to the Optimal Control on Flows of Probability Measures}
\label{subsec:lq_model_ctrl_flow}

We now proceed to reduce the principal's optimal contracting problem to a deterministic control problem on the flow of the probability measures corresponding to a continuous-time Markov chain. From the SDE (\ref{eq:forward_y_mkv1}) and the definition of $\hat H$, we have:

\begin{align*}
Y^{(\bZ,\br,Y_0)}_T =&\;\; Y_0 - \int_0^T [c(t, X_t,\hat a(X_{t-}, Z_t), p^{(\bZ,\br,Y_0)}(t)) - u(r_t)] dt+ \int_0^T Z_t^* d\mathcal{M}_t^{(\bZ,\br,Y_0)},
\end{align*}
where $\\bcM^{(Z,r,Y_0)}$ is a martingale under the measure $\mathbb{P}^{(\bZ,\br,Y_0)}$. Now using $U^{-1}(y) = y$ and injecting the SDE into the objective function of the principal defined in (\ref{eq:principal_pb}), we may rewrite the principal's optimal contracting problem:

\begin{align*}
V(\kappa)
 =&  \inf_{\mathbb{E}^{\mathbb{P}}[Y_0] \le \kappa} \inf_{\substack{\bZ\in\mathcal{H}_{X}^2\\ \br \in \mathcal{R}}}\mathbb{E}^{(\bZ,\br,Y_0)}\bigg[\int_0^T \big[c_0(t, p^{(\bZ,\br,Y_0)}(t)) + c(t, X_t,\hat a(X_{t-}, Z_t), p^{(\bZ,\br,Y_0)}(t))\big] dt\\ 
 &\quad\quad\quad\quad\quad\quad\quad\quad\quad\quad+\int _0^T \big[r_t - u(r_t)\big] dt +C_0(p^{(\bZ,\br,Y_0)}(T)) - Y_0\bigg]\\
=& -\kappa + \inf_{\substack{\bZ\in\mathcal{H}_{X}^2\\ \br \in \mathcal{R}}}\mathbb{E}^{(\bZ,\br,Y_0)}\bigg[\int_0^T \big[c_0(t, p^{(\bZ,\br,Y_0)}(t)) + c(t, X_t,\hat a(X_{t-}, Z_t), p^{(\bZ,\br,Y_0)}(t))\big] dt\\
 &\quad\quad\quad\quad\quad\quad\quad\quad\quad\quad+\int _0^T \big[r_t - u(r_t)\big] dt +C_0(p^{(\bZ,\br,Y_0)}(T))\bigg].
\end{align*}
Here we have used the equality $\mathbb{E}^{\mathbb{P}}[Y_0] = \mathbb{E}^{(\bZ,\br,Y_0)}[Y_0]$. Notice that both the transition rate matrix $Q$ and the optimizer $\hat a$ do not depend on the reward $\br$ or the agent's expected total cost $Y_0$, so we can drop the dependency of $p^{(\bZ,\br,Y_0)}$, $\mathbb{P}^{(\bZ,\br,Y_0)}$ and $\alpha^{(\bZ,\br,Y_0)}$ on $\br$ and $Y_0$. We can also isolate the optimal choice of $\br$ from the principal's optimization problem. Indeed, let $\hat r$ be a minimizer of the mapping $r \rightarrow r - u(r)$. It is immediately clear that for the principal it is optimal to choose $r_t = \hat r$ for $t\in [0, T]$. It follows that:

\begin{align*}
V(\kappa) =&\inf_{\bZ\in\mathcal{H}_{X}^2}\mathbb{E}^{\bZ}\bigg[\int_0^T \big[c_0(t, p^{\bZ}(t)) + \sum_{i=1}^m\mathbbm{1}(X_t = e_i)[c_1(t, e_i ,p^{\bZ}(t))  + \frac{\gamma_i}{2} \hat a^2(e_i, Z_t)]\big] dt\\
&\hskip 50pt+C_0(p^{\bZ}(T))\bigg] -\kappa + T (\hat r - u(\hat r)).
\end{align*}
We now focus on:

$$
W :=\inf_{\bZ\in\mathcal{H}_{X}^2}\mathbb{E}^{\bZ}\bigg[\int_0^T \big[c_0(t, p^{\bZ}(t)) + \sum_{i=1}^m\mathbbm{1}(X_t = e_i)[c_1(t, e_i ,p^{\bZ}(t))  + \frac{\gamma_i}{2} \hat a^2(e_i, Z_t)]\big] dt +C_0(p^{\bZ}(T))\bigg],
$$
where $p^{\bZ}(t) = \mathbb{P}^{\bZ}[X_t]$ and under $\mathbb{P}^{\bZ}$, $\bX$ has the following decomposition with $\bcM^{(Z)}$ being a $\mathbb{P}^{\bZ}$-martingale:
\[
X_t = X_0 + \int_0^t Q^*(s, \hat a(X_{s-}, Z_s), p^{\bZ}(s)) X_{s-} ds + \mathcal{M}_t^{\bZ}.
\]
The key observation is that the control $\bZ$ affects the value function only through the optimal control $\alpha^{\bZ}_t := \hat a(t, X_{t-}, Z_t)$ and the mapping $\mathcal{H}_X^2 \ni \bZ \rightarrow \alpha^{\bZ} \in \mathbb{A}$ is surjective. For any $\bZ \in \mathcal{H}_X^2$, it is clear from the definition of $\hat a$ that $\alpha^{\bZ} \in \mathbb{A}$. On the other hand, given an arbitrary $\balpha \in \mathbb{A}$, we set the $j$-th component of the process $\bZ$ to be:
\begin{equation}\label{eq:conversion_z2alpha}
Z^j_t := \sum_{i,i\neq j}\mathbbm{1}(X_{t-} = e_i) \frac{\gamma_i\alpha_t}{(m-1)\sum_{k,k\neq i}\lambda_{i,k}}.
\end{equation}
By the boundedness of $\balpha$ we have $\bZ \in \mathcal{H}_X^2$ and it is plain to verify that $\hat a(t, X_{t-}, Z_t) = \alpha_t$. Therefore we can transform the optimization problem $W$ into:
\begin{equation}\label{eq:w}
W = \inf_{\balpha\in\mathbb{A}}I(\balpha),
\end{equation} 
where we define:
\begin{equation}\label{eq:principal_pb_tr1}
I(\balpha):=\mathbb{E}^{\balpha}\bigg[\int_0^T \big[c_0(t, p^{\balpha}(t)) + \sum_{i=1}^m \mathbbm{1}(X_t = e_i)[c_1(t, e_i ,p^{\balpha}(t))  + \frac{\gamma_i}{2} \alpha^2_t]\big]\;dt +C_0(p^{\balpha}(T))\bigg].
\end{equation}
Here, with a mild abuse of notation, we denote by $\mathbb{E}^{\balpha}$ the expectation under $\mathbb{P}^{\balpha)}$, and $\mathbb{P}^{\balpha}$ is the probability measure defined by the following McKean-Vlasov SDE:
\begin{align}
d\mathcal{E}_t^{\balpha} =&\;\; 1 + \int_0^t \mathcal{E}_{s-}X_{s-}^* (Q(s, \alpha_s, p^{\balpha\balpha}(s)) - Q^0)\psi_s^+ d\mathcal{M}_s,\label{eq:principal_pb_tr1_dyn_eq1}\\
\frac{d\mathbb{P}^{\balpha}}{d\mathbb{P}} =&\;\; \mathcal{E}_T^{\balpha},\quad p^{\balpha}(t) = \mathbb{E}^{\balpha}[X_t].\label{eq:principal_pb_tr1_dyn_eq2}
\end{align}
Recall that under $\mathbb{P}^{\balpha}$, $\bX$ has the decomposition:
\[
X_t = X_0 + \int_0^t Q^*(s, \alpha_s, p^{\balpha}(s)) X_{s-} ds + \mathcal{M}^{\balpha}_t,
\]
where $\bcM^{\balpha}$ is a $\mathbb{P}^{\balpha}$-martingale. Taking expectation under $\mathbb{P}^{\balpha}$, we have:
\begin{align*}
p^{\balpha}(t) =&\;\;p^{\balpha}(0) + \int_0^t \mathbb{E}^{\balpha}[Q^*(s, \alpha_s, p^{\balpha}(s)) X_{s-}]\; ds\\
 =&\;\;p^{\circ} + \int_0^t \sum_{j=1}^m p_j^{\balpha}(s) \mathbb{E}^{\balpha}[Q^*(s, \alpha_s, p^{\balpha}(s)) \cdot e_j | X_{s-} = e_j]\;ds.
 \end{align*}
Rewriting this in terms of the coordinates $p_i^{\balpha}(t)$ of $p^{\balpha}(t)$ and using the linear structure of $Q$, we have:
\begin{align*}
p_i^{\balpha}(s)=&\;\; p^{\circ}_i + \int_0^t \sum_{j=1}^m p_j^{\balpha}(s) \mathbb{E}^{\balpha}[e_j^*\cdot Q(s, \alpha_s, p^{\balpha}(s))  e_i | X_{s-} = e_j]\;ds \\
=&\;\;p^{\circ}_i + \int_0^t \sum_{j=1}^m p_j^{\balpha}(s) \mathbb{E}^{\balpha}[q(s,j,i,\alpha_s,p^{\balpha}(s)) | X_{s-} = e_j]\;ds \\
= &\;\; p^{\circ}_i + \int_0^t \sum_{j,j\neq i} p_j^{\balpha}(s) \mathbb{E}^{\balpha}[q(s,j,i,\alpha_s,p^{\balpha}(s)) | X_{s-} = e_j]\;ds \\
& \hskip 50pt - \int_0^t \sum_{j,j\neq i} p_i^{\balpha}(s) \mathbb{E}^{\balpha}[q(s,i,j,\alpha_s,p^{\balpha}(s)) | X_{s-} = e_i]\;ds \\
=&\;\; p^{\circ}_i + \int_0^t \sum_{j,j\neq i} p_j^{\balpha}(s) \left[\lambda_{j,i}(\mathbb{E}^{\balpha}[\alpha_s | X_{s-} = e_j] - \ubar\alpha) + \bar q_{j,i}(s,p^{\balpha}(s))\right]\;ds\\
&\hskip 20pt  - \int_0^t \sum_{j,j\neq i} p_i^{\balpha}(s) \left[\lambda_{i,j}(\mathbb{E}^{\balpha}[\alpha_s | X_{s-} = e_i] - \ubar\alpha) + \bar q_{i,j}(s,p^{\balpha}(s))\right]\;ds.
\end{align*}
We see that the dynamics of $p^{\balpha}$ are completely driven by the deterministic processes $t\rightarrow\mathbb{E}^{\balpha}[\alpha_t|X_{t-} = e_i]$ for $i\in \{1,\dots,m\}$.  From now on, we denote $\tilde\alpha_t^i:=\mathbb{E}^{\balpha}[\alpha_t|X_{t-} = e_i]$ and $\tilde\alpha_t := [\tilde\alpha_t^1,\dots, \tilde\alpha_t^m]$. By Jensen's inequality, for all $\balpha\in\mathbb{A}$ we have:
\begin{align*}
I(\balpha) =&\;\;\int_0^T \bigg[c_0(t,p_t^{\balpha}) + \sum_{i=1}^m \big( c_1(t,e_i, p_t^{\balpha}) +  \frac{\gamma_i}{2} \mathbb{E}^{\balpha}[\alpha^2_t|X_{t-} = e_i]\big)p_i^{\balpha}(t) \bigg] dt\\
&\;\; + C_0(p^{\balpha}(T))\\
\ge&\;\;\int_0^T\bigg[c_0(t,p_t^{\balpha}) + \sum_{i=1}^m \big(c_1(t,e_i, p_t^{\balpha}) +  \frac{\gamma_i}{2} (\mathbb{E}^{\balpha}[\alpha_t|X_{t-} = e_i])^2\big)p_i^{\balpha}(t) \bigg] dt + C_0(p^{(\alpha)}(T)).
\end{align*}
This leads to the deterministic control problem:

\begin{equation}
\label{eq:w_tilde}
\tilde W := \inf_{\tilde{\balpha}\in\tilde{\mathbb{A}}} \tilde I(\tilde\balpha),
\end{equation}
where we denote by $\tilde{\mathbb{A}}$ the collection of all measurable mappings from $[0,T]$ to $A^m$, and $\pi^{\tilde\balpha}$ the solution to the following system of coupled ODEs:

\begin{equation}\label{eq:principal_pb_det_dyn}
\begin{aligned}
\frac{d\pi_i^{\tilde\balpha}(t)}{dt} = &\;\; \sum_{j,j\neq i} \pi_j^{\tilde\balpha}(t) \left[\lambda_{j,i}(\tilde\alpha_t^j - \ubar\alpha) + \bar q_{j,i}(t,\pi^{\tilde\balpha}(t))\right]
-\sum_{j,j\neq i}\pi_i^{\tilde\balpha}(t) \left[\lambda_{i,j}(\tilde\alpha_t^i - \ubar\alpha) + \bar q_{i,j}(t,\pi^{\tilde\balpha}(t))\right],\\
\pi^{\tilde\balpha}(0) =&\;\; p^{\circ},
\end{aligned}
\end{equation}
and finally the objective function $\tilde I$ is given by:

\begin{equation}
\label{eq:principal_pb_det_obj}
\tilde I(\tilde\balpha) :=\int_0^T\bigg[c_0(t,\pi^{\tilde\balpha}(t)) + \sum_{i=1}^m [c_1(t,e_i, \pi^{\tilde\balpha}(t)) + \frac{\gamma_i}{2} (\tilde\alpha_t^i)^2]\pi_i^{\tilde\balpha}(t) \bigg] dt + C_0(\pi^{\tilde\balpha}(T)).
\end{equation}
The following results show that $W$ and $\tilde W$ are two equivalent optimization problems.

\begin{proposition}
\label{prop:optimal_contracting}
We have $W = \tilde W$. If $\tilde\balpha$ is a solution to the optimization problem $\tilde W$, then the predictable process $\balpha$ defined by $\alpha_t = \sum_{i=1}^m\mathbbm{1}(X_{t-}=e_i)\tilde\alpha_t^i$ is an optimal control for $W$. Moreover, under the probability measure at the optimum, the agent's state evolves as a continuous-time Markov chain. 
\end{proposition}

\begin{proof}
From the derivation above, we already see that $W \ge \tilde W$. We now show that $\tilde W \ge W$. Given any $\tilde\balpha \in\tilde{\mathbb{A}}$, we set:
\begin{equation}\label{eq:optimal_control_transform}
\alpha_t := \sum_{i=1}^m\mathbbm{1}(X_{t-}=e_i)\tilde\alpha_t^i.
\end{equation}
Clearly, $\balpha \in \mathbb{A}$. By standard results on ordinary differential equation, equation (\ref{eq:principal_pb_det_dyn}) admits a unique solution, say  $(\pi^{\tilde\balpha}(t))_{t\in [0,T]}$. Now let us consider the probability measure $\tilde{\mathbb{P}}$ defined by:

\begin{align}
\frac{d\tilde{\mathbb{P}}}{d\mathbb{P}} =&\;\; \tilde{\mathcal{E}}_T,\label{eq:principal_pb_det_dyn2_eq1}\\
d\tilde{\mathcal{E}}_t =&\;\; \tilde{\mathcal{E}}_{t-}X_{t-}^* (Q(t, \alpha_t, \pi^{(\tilde\alpha)}(t)) - Q^0)\psi_t^+ d\mathcal{M}_t,\;\;\tilde{\mathcal{E}}_0= 1.\label{eq:principal_pb_det_dyn2_eq2}
\end{align}
It is easy to see that under $\tilde{\mathbb{P}}$, the canonical process $\bX$ has the decomposition:

\[
X_t = X_0 + \int_0^t Q^*(s, \alpha_s, \pi^{\tilde\balpha}(s))X_{s-} ds + \tilde{\mathcal{M}}_t = X_0 + \int_0^t \tilde Q^*(s)X_{s-} ds + \tilde{\mathcal{M}}_t,
\]
where $\tilde\bcM=(\tilde{\mathcal{M}}_t)_{0\le t\le T}$ is a martingale and $\tilde Q(t)$ is the transition rate matrix with components $\tilde Q_{ij}(t) = q(t,i,j,\tilde\alpha_t^i,\pi^{\tilde\balpha}(t))$. This implies that $\bX$ is a continuous-time Markov chain under $\tilde{\mathbb{P}}$ and it is then straightforward to write the Kolmogorov equation satisfied by the marginal laws of $\bX$ under $\tilde{\mathbb{P}}$. Comparing this Kolmogorov equation with the ODE (\ref{eq:principal_pb_det_dyn}), we conclude by the uniqueness of the solutions that $\tilde{\mathbb{E}}[X_{t}]= \pi^{\tilde\balpha}(t)$, where $\tilde{\mathbb{E}}$ stands for the expectation under $\tilde{\mathbb{P}}$. Now in light of equations (\ref{eq:principal_pb_det_dyn2_eq1})-(\ref{eq:principal_pb_det_dyn2_eq2}), we conclude that $\tilde{\mathbb{P}}$ is the solution to the McKean-Vlasov SDE defined in (\ref{eq:principal_pb_tr1_dyn_eq1})-(\ref{eq:principal_pb_tr1_dyn_eq2}) corresponding to the control $\balpha$. It follows that $\tilde{\mathbb{P}} = \mathbb{P}^{\balpha}$ and $\pi^{\tilde\balpha}(t) = p^{\balpha}(t)$. We can now compute $I(\balpha)$:

\begin{align*}
I(\balpha)
=&\;\;\mathbb{E}^{\balpha}\left[\int_0^T \big(c_0(t, p^{\balpha}(t)) + \sum_{i=1}^m \mathbbm{1}(X_t = e_i)[c_1(t, e_i ,p^{\balpha}(t))  + \frac{\gamma_i}{2} \alpha^2_t] \big)dt + C_0(p^{\balpha}(T))\right]\\
=&\;\;\int_0^T \left[c_0(t, p^{\balpha}(t)) + \sum_{i=1}^m [c_1(t, e_i, p^{\balpha}(t)) +  \frac{\gamma_i}{2} (\tilde\alpha^i_t)^2]\mathbb{P}^{\balpha}[X_{t} = e_i] \right] dt + C_0(p^{\balpha}(T))\\
=&\;\; \int_0^T\bigg[c_0(t,\pi^{\tilde\balpha}(t)) + \sum_{i=1}^m [c_1(t,e_i,\pi^{\tilde\balpha}(t))  +  \frac{\gamma_i}{2}(\tilde\alpha_t^i)^2] \pi_i^{\tilde\balpha}(t) \bigg]dt + C_0(\pi^{\tilde\balpha}(T))\\
=&\;\;\tilde I(\tilde\balpha).
\end{align*}
From this, we deduce that $\tilde I(\tilde\balpha)  = I(\balpha) \ge W$ and finally $\tilde W \ge W$. Therefore we have $\tilde W = W$. Let $\tilde\balpha \in \tilde{\mathbb{A}}$ be the optimizer of $\tilde W$ and define $\balpha\in\mathbb{A}$ as in (\ref{eq:optimal_control_transform}). Then from the computations above, we see that $\tilde W = \tilde I(\tilde\balpha)  = I(\balpha)$. This immediately implies $I(\balpha) = W$ and $\balpha$ is the optimal control of $W$.
\end{proof}

\subsection{Construction of the Optimal Contract}

We continue to investigate the deterministic control problem $\tilde W$ as defined in equations (\ref{eq:w_tilde})-(\ref{eq:principal_pb_det_obj}). Once we have identified the optimal strategy of the agents at the equilibrium from the control problem $\tilde W$, we can then provide a semi-explicit construction for the optimal contract.
\begin{lemma}\label{lem:existence_optimal_control_deterministic}
An optimal control exists for the deterministic control problem $\tilde W$.
\end{lemma}
\begin{proof}
It is straightforward to verify that: (1) The space of controls is convex and compact. (2) The right-hand side of the ODE (\ref{eq:principal_pb_det_dyn}) is $\mathcal{C}^1$ and is linear in $\tilde \alpha$. (3) The running cost is $\mathcal{C}^1$ and convex in $\alpha$ for all $(t,\pi)\in [0,T]\times\mathcal{S}$ and the terminal cost is $\mathcal{C}^1$. This allows us to apply Theorem I.11.1 in \cite{fleming2006controlled} and obtain the existence of the optimal control.
\end{proof}

Having verified that $\tilde W$ admits an optimal solution, we now apply the necessary part of the Pontryagin maximum principle (see Theorem I.6.3 \cite{fleming2006controlled}), and derive a system of ODEs that characterizes the optimal control and the corresponding flow of probability measures. The Hamiltonian $\tilde H$ of the control problem $\tilde W$ is a mapping from $[0,T]\times\mathcal{S}\times\mathbb{R}^m\times A^m$ to $\mathbb{R}$ defined by:
\begin{align*}
\tilde H(t,\pi,y,\alpha) :=&\sum_{i=1}^m \sum_{j,j\neq i} y_i\left[\pi_j \left(\lambda_{j,i}(\alpha_j - \ubar\alpha) + \bar q_{j,i}(t,\pi)\right) - \pi_i \left(\lambda_{i,j}(\alpha_i - \ubar\alpha) + \bar q_{i,j}(t,\pi)\right)\right]\\
&+ c_0(t,\pi) + \sum_{i=1}^m \pi_i\left[c_1(t,e_i,\pi)+\frac{\gamma_i}{2}(\alpha_i)^2\right].
\end{align*}
It is straightforward to obtain:
\[
\partial_{\alpha_i}\tilde H(t,\pi,y,\alpha) = \pi_i \left(\sum_{k\neq i}(y_k - y_i)\lambda_{i,k} + \gamma_i\alpha_i\right),
\]
and the minimizer of $\alpha\rightarrow \tilde H(t,\pi,y,\alpha)$ is
\[
\hat a_i(y) = b(-\sum_{k\neq i}\lambda_{i,k}(y_k - y_i) / \gamma_i),
\]
for the function $b$ we already defined as $b(z) := \min\{\max\{z,\ubar\alpha\},\bar\alpha\}$. By the necessary condition of Pontryagin's maximum principle, if $(\pi(t))_{0\le t\le T}$ is the flow of measures associated with the optimal control, then $(\pi(t), y(t))_{t\in [0,T]}$ is the solution to the following system of forward-backward ODEs:

\begin{align}
\frac{d\pi(t)}{dt} =&\;\; \partial_y\tilde H(t,\pi(t),y(t),\hat a(y(t))),\;\;\pi(0) =p^{\circ},\label{eq:pontryagin_lq_1}\\
\frac{dy(t)}{dt}=&\;\;-\partial_\pi\tilde H(t,\pi(t),y(t),\hat a(y(t))),\;\;y(T) = \nabla C_0(\pi(T)).\label{eq:pontryagin_lq_2}
\end{align}
Summarizing the above discussion, as well as the arguments which allow us to reduce the optimal contracting problem to the deterministic control problem we just solved, we provide a semi-explicit construction of the optimal contract and the optimal strategy of the agents at the equilibrium.

\begin{theorem}
Let $(\hat\pi, \hat y)$ be the solution to the system (\ref{eq:pontryagin_lq_1})-(\ref{eq:pontryagin_lq_2}), and let us define the processes $\hat\balpha\in\mathbb{A}$ and $\hat \bZ \in \mathcal{H}^2_{X}$ by:

\begin{align}
\hat\alpha_t :=&\;\; \sum_{i=1}^m \mathbbm{1}(X_{t-} = e_i)\hat a_i(\hat y(t)),\\
\hat Z_t :=&\;\; [Z^1_t, \dots, Z^m_t]^*,\quad \hat Z^i_t := \sum_{j,j\neq i}\mathbbm{1}(X_{t-} = e_j) \frac{\gamma_j\hat a_j(\hat y(t))}{(m-1)\sum_{k\neq j}\lambda_{j,k}}.
\end{align}
Now let $y_0 \in \mathbb{R}^m$ be such that $p^{\circ}_0\cdot y_0 \le \kappa$ and let $\hat r$ be the minimizer of the mapping $r\rightarrow r- u(r)$. We then define the random variable $\hat \xi$ almost surely by the following Stieltjes integral:

\begin{equation}
\label{eq:optimal_contract_explicit}
\hat \xi := - X_0^* y_0 + \int_0^T [c(t, X_{t-}, \hat\alpha_t, \hat \pi(t))-\hat r + X^*_{t-} Q(t, \hat\alpha_t, \hat \pi(t))\hat Z_t ] dt - \int_0^T \hat Z_t^* dX_{t-}.
\end{equation}
Then $(\hat \br, \hat \xi)$ is an optimal contract. Moreover, under the optimal contract, every agent adopts the Markovian strategy  where they pick the control $\hat a_i(\hat y(t))$ when in the state $e_i$ at time $t$, and the flow of distributions of agents' states is given by $(\hat\pi(t))_{0\le t\le T}$.
\end{theorem}

\section{\textbf{Application to a Model of Epidemic Containment}}
\label{sec:epidemic}

To illustrate the inner workings of the model completely solved above, we consider an example of epidemic containment. We imagine a disease control authority that aims at containing the spread of a virus over a time period $[0,T]$ within its own jurisdiction, which consists of two cities $A$ and $B$. The state of each individual is encoded by whether it is infected (denoted by $I$) or healthy (denoted by $H$), and by its location (denoted by $A$ or $B$). Therefore the state space is $E = \{AI, AH, BI, BH\}$, and we use $\pi_{AI}, \pi_{AH}, \pi_{BI}, \pi_{BH}$ to denote the proportion of individuals in each of these $4$ states. We assume that each individual's state evolves as a continuous-time Markov chain, and our modeling of the transition rate accounts for the following set of mechanisms regarding the contraction of the virus and the possible migration of individuals between the two cities:

\vspace{2mm}
\noindent (1) Within each city, the rate of contracting the virus depends on the proportion of infected individuals in the city, and accordingly we assume that the transition rate from state $AH$ to state $AI$ is $\theta_A^-(\frac{\pi_{AI}}{\pi_{AI} + \pi_{AH}})$, while the transition rate from state $BH$ to state $BI$ is $\theta_B^-(\frac{\pi_{BI}}{\pi_{BI} + \pi_{BH}})$. Here $\theta_A^-$ and $\theta_B^-$ are two increasing, positive and differentiable mappings from $[0,1]$ to $\mathbb{R}^+$. They capture the quality of health care in city $A$ and $B$, respectively.

\vspace{2mm}
\noindent (2) Likewise, the rate of recovery in each city is a function of the proportion of healthy individuals in the city. We thus assume that the transition rate from state $AI$ to state $AH$ is $\theta_A^+(\frac{\pi_{AH}}{\pi_{AI} + \pi_{AH}})$ and the transition rate from state $BH$ to state $BI$ is $\theta_B^+(\frac{\pi_{BH}}{\pi_{BI} + \pi_{BH}})$. Similarly, $\theta_A^+$ and $\theta_B^+$ are two increasing, positive and differentiable mappings from $[0,1]$ to $\mathbb{R}^+$, characterizing the quality of health care in city $A$ and $B$ respectively.

\vspace{2mm}
\noindent (3) Each individual can choose a level of effort $\alpha$ to move to the other city. We assume that the efficacy of that effort depends on whether the individual is healthy or infected. Accordingly, we set $\nu_I\alpha$ as the transition rates between the states $AI$ and $BI$, and we set $\nu_H\alpha$ as the transition rates between the states $AH$ and $BH$. 

\vspace{2mm}
\noindent (4) To model the inflow of infection, we assume that each individual's status of infection does not change when it moves between cities. This means that we set the transition rates between the state $AI$ and $BH$, and the transition rates between the state $AH$ and $BI$ to $0$.

\vskip 6pt
To summarize, we define the transition rate matrix $Q$ to be:
\begin{equation}\label{eq:epicontain_q_matrix}
Q(t,\alpha,\pi):=\left[
\begin{array}{cccc}
\cdots & \theta_A^+(\frac{\pi_{AH}}{\pi_{AI} + \pi_{AH}}) & \nu_I \alpha & 0\\
\theta_A^-(\frac{\pi_{AI}}{\pi_{AI} + \pi_{AH}}) & \cdots & 0 & \nu_H\alpha\\
\nu_I\alpha & 0 & \cdots & \theta_B^+(\frac{\pi_{BH}}{\pi_{BI} + \pi_{BH}}) \\
0 & \nu_H\alpha & \theta_B^-(\frac{\pi_{BI}}{\pi_{BI} + \pi_{BH}}) & \cdots
\end{array}
\right].
\end{equation}
Here the transition rate matrix is written assuming the order $AI, AH, BI, BH$ for the states. For simplicity of the notation, we also omit the diagonal elements. Indeed, since $Q$ is a transition rate matrix, each diagonal element equals the negative of the sum of the off-diagonal elements in the same row.

\vspace{2mm}
We resume the description of our model in terms of the cost functions for the individuals and the disease control authority. We assume that individuals living in the same city incur the same cost, which depends on the proportion of infected individuals in that city. On the other hand, the cost for exerting the effort to move depends on the status of infection. Using the notations of the cost function for the linear quadratic model as in (\ref{eq:agent_cost_lq}), we define:
\begin{align}
c_1(t, AI, \pi) =&\;\; c_1(t, AH, \pi) := \phi_A\left(\frac{\pi_{AI}}{\pi_{AI} + \pi_{AH}}\right),\label{eq:epicontain_cost_A}\\
c_1(t, BI, \pi) =&\;\; c_1(t, BH, \pi) := \phi_B\left(\frac{\pi_{BI}}{\pi_{BI} + \pi_{BH}}\right),\label{eq:epicontain_cost_B}\\
\gamma_{AI} =&\;\; \gamma_{BI} := \gamma_I,\;\;\gamma_{AH}=\gamma_{BH} := \gamma_H,\label{eq:epicontain_cost_move}
\end{align}
where $\phi_A$ and $\phi_B$ are two increasing mappings on $\mathbb{R}$. For the authority, we propose to use the following running cost and terminal cost:
\begin{align}
c_0(t,\pi) =&\;\; \exp(\sigma_A \pi_{AI} + \sigma_B \pi_{BI}),\label{eq:epicontain_cost_authority1}\\
C_0(\pi) =&\;\; \sigma_P\cdot(\pi_{AI} + \pi_{AH} - \pi_A^0)^2,\label{eq:epicontain_cost_authority2}
\end{align}
where $\pi_A^0$ is the population of city $A$ at time $0$. Intuitively speaking, the above cost function tries to encapsulate a form of trade-off between the control of the epidemic  and  population planning. On the one hand, the authority attempts to minimize the infection rate of both cities. On the other hand, as we shall see shortly in the numerical simulation, individuals tends to move away from the city with higher infection rate and poorer health care, which might result in the overpopulation of the other city. Therefore, the authority also wishes to maintain the population of both cities at a steady level. The coefficients $\sigma_A$, $\sigma_B$ and $\sigma_P$ reflects the relative importance the authority attributes to each of these objectives.

\vspace{2mm}
It can be easily verified that the setup outlined above satisfies the assumptions of the linear quadratic model studied in Section \ref{sec:lq_model}. Although the forward-backward system of ODEs (\ref{eq:pontryagin_lq_1}) - (\ref{eq:pontryagin_lq_2}) which characterizes the optimal contract can be readily derived, for the sake of completeness, we shall give the details of the equations to be solved in the appendix.

\vspace{2mm}
For the purpose of illustration, we give the results of numerical simulations for a scenario in which city $A$ has a higher quality of health care than city $B$. Accordingly we set:
\begin{equation}\label{eq:infection_rate_simulation}
\theta^+_A(q) := 0.4 q,\;\;\theta^-_A(q) := 0.1 q,\;\;\theta^+_B(q) := 0.2 q,\;\;\theta^-_B(q) := 0.2 q.
\end{equation}
This means that it is easier to recover and harder to get infected in city $A$ than in city $B$. In addition, we assume that individuals suffer a higher cost associated with the epidemic in city $B$ than in city $A$, and we set the cost function of individuals in each city to be:
\begin{equation}\label{eq:individual_cost_simulation}
\phi_A(q) := q,\;\;\phi_B(q):=2 q.
\end{equation}
We set the maximal possible effort of individuals to $\bar\alpha := 10$ and the coefficients for quadratic cost of efforts to $\gamma_I := 2.0$ and $\gamma_H = 0.5$. Finally, the parameters for the cost of the authority in equations (\ref{eq:epicontain_cost_authority1}) and (\ref{eq:epicontain_cost_authority2}) are set to:
\begin{equation}\label{eq:authority_cost_parameters}
\sigma_A = \sigma_B := 1,\;\;\sigma_P := 0.
\end{equation}
Notice that the authority gives the same importance to the infection rates of city $A$ and city $B$ while disregarding the problem of overpopulation.

\vspace{2mm}
In the following, we shall visualize the effect of the disease control authority's intervention by comparing the equilibrium computed from the principal agent problem with the equilibrium from the mean field game of \emph{anarchy}. By the term \emph{anarchy}, we refer to the situation where the states of individuals are still governed by the same transition rate, but the individuals do not receive any rewards or penalties from the authority. More specifically, the expected total cost of each individual is given by:
\[
\mathbb{E}^{\mathbb{Q}^{(\balpha,\pi)}}\left[\int_0^T c(t, X_t, \alpha_t, \pi(t)) dt\right],
\]
with the instantaneous cost $c$ is given by:
\begin{align*}
c(t, x, \alpha, \pi) := &\;\;\mathbbm{1}(x\in\{AI, AH\})\cdot\phi_A\left(\frac{\pi_{AI}}{\pi_{AI} + \pi_{AH}}\right)\\
&\;\;+ \mathbbm{1}(x\in\{BI, BH\})\cdot\phi_B\left(\frac{\pi_{BI}}{\pi_{BI} + \pi_{BH}}\right)\\
&\;\;+ (\mathbbm{1}(x\in\{AI, BI\})\gamma_I + \mathbbm{1}(x\in\{AH, BH\})\gamma_H)\cdot\frac{\alpha^2}{2}.
\end{align*}
Following the analytical approach of finite-state mean field games introduced in \cite{gomes2013discrete}, it is straightforward to derive the system of forward-backward ODEs characterizing the Nash equilibrium (See the system of ODEs (12)-(13) in \cite{gomes2013discrete}). For the sake of completeness, we display the system of ODEs in the appendix.

\vspace{2mm}
The effect of the authority's intervention is conspicuous in diminishing the infection rate among the entire population, as is shown in the upper panel of Figure \ref{fig:infection_rate_plot}. However, when we visualize the respective infection rate of each city in the lower panels of Figure \ref{fig:infection_rate_plot}, we do observe a surge of infection in city $A$ as a result of the authority's intervention, although the epidemic eventually dies down. This is due to the inflow of infected individuals from city $B$ in the early stage of the epidemic. Indeed, since city $A$ provides better health care and maintains a lower rate of infection compared to city $B$, an infected individual has a better chance of recovery if it moves to city $A$. The cost of moving prevents individuals from seeking better health care in the scenario of anarchy, whereas individuals seem to receive subsidy to move when the authority tries to intervene. This outcome of the authority's intervention is further corroborated when we visualize an individual's optimal strategy in Figure \ref{fig:optimal_strat_plot}. We observe that individuals have a greater propensity to move when the authority provides incentives.
\begin{figure}[h]
\centering
\includegraphics[scale=0.5, trim = 0mm 0mm 0mm 0mm, clip=true]{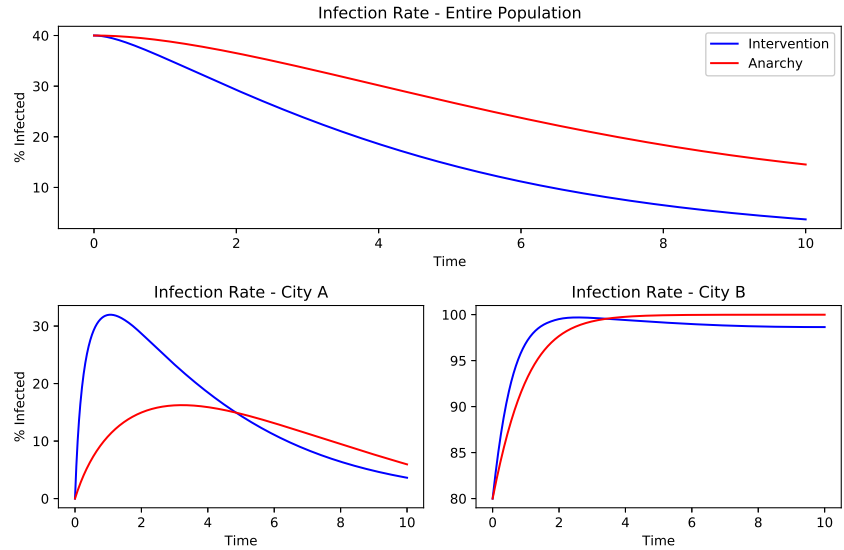}
\caption{Evolution of infection rate with and without authority's intervention.}
\label{fig:infection_rate_plot}
\end{figure}

\begin{figure}[h]
\centering
\includegraphics[scale=0.5, trim = 0mm 0mm 0mm 0mm, clip=true]{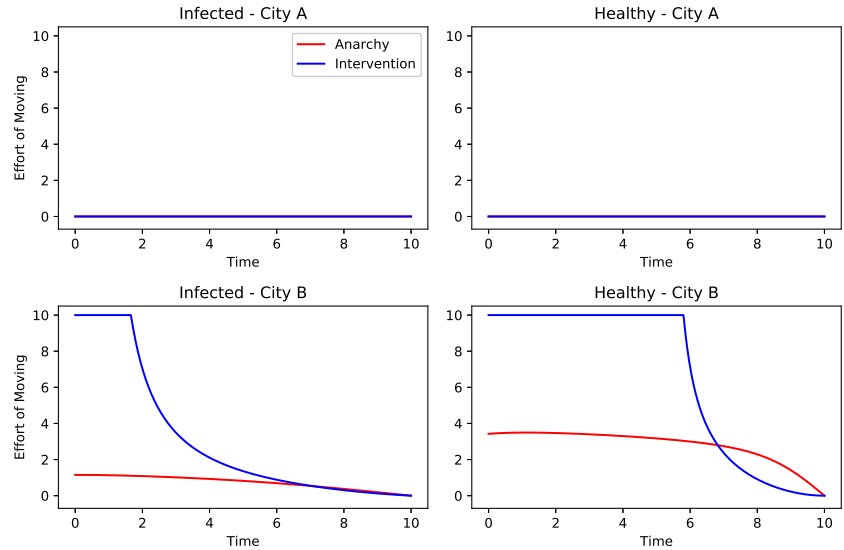}
\caption{Optimal effort of moving for an individual in different states.}
\label{fig:optimal_strat_plot}
\end{figure}
Recall that in the above numerical computation, by setting $\sigma_P = 0$ in its terminal cost function (\ref{eq:epicontain_cost_authority2}), the authority does not attempt to maintain the balance of population between the two cities. We now investigate how the behavior of the individuals changes when the authority seeks to prevent the occurrence of overpopulation. To this end, we rerun the computation with $\sigma_P = 1.5$ and all the other parameters unchanged. In Figure \ref{fig:population_planning}, we compare the evolution of the population in city $A$ as well as the total infection rate with and without the population planning. When the authority does not try to control the flow of the population, the entire population ends up in city $A$. However, when a terminal cost related to population planning is introduced, we see a more balanced population distribution while the infection rate is still well managed. This can be explained by Figure \ref{fig:optimal_strat_plot_planning}, from which we have a more detailed perspective on the change of individual behavior when the authority implements the population planning. We see that healthy individuals in city $A$ are now encouraged to move the city $B$, in order to compensate the exodus caused by the epidemic in city $B$. On the other hand, healthy individuals in city $B$ are now incentivized to stay in place.
\begin{figure}[H]
\centering
\includegraphics[scale=0.5, trim = 0mm 0mm 0mm 0mm, clip=true]{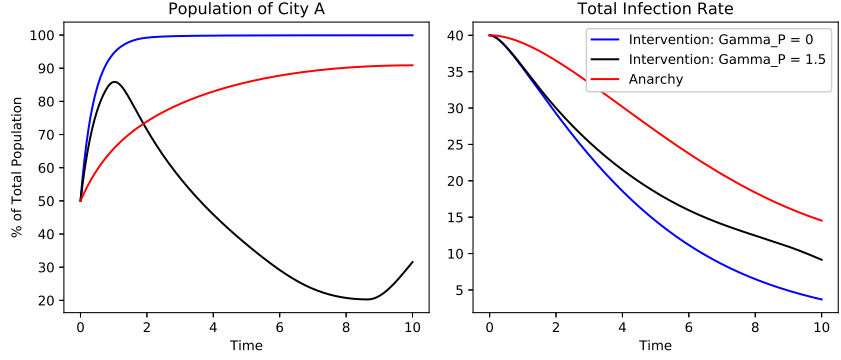}
\caption{Evolution of population in city $A$ and the total infection rate. The blue curve corresponds to intervention without population planning ($\sigma_P = 0$), the black curve corresponds to intervention with population planning ($\sigma_P = 1.5$) and the red curve corresponds to the absence of intervention.}
\label{fig:population_planning}
\end{figure}

\begin{figure}[H]
\centering
\includegraphics[scale=0.5, trim = 0mm 0mm 0mm 0mm, clip=true]{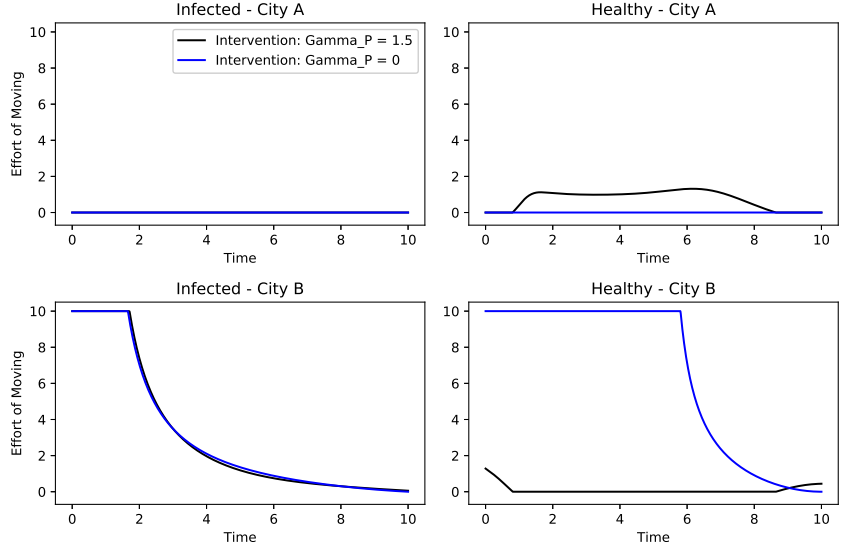}
\caption{Individual's optimal effort of moving with and without population planning.}
\label{fig:optimal_strat_plot_planning}
\end{figure}

\section{\textbf{Appendix: System of ODEs for Epidemic Containment}}

\subsection{Authority's optimal planning}

Using the transition rate (\ref{eq:epicontain_q_matrix}) and the cost functions (\ref{eq:epicontain_cost_A})-(\ref{eq:epicontain_cost_authority2}), the system of ODEs (\ref{eq:pontryagin_lq_1})-(\ref{eq:pontryagin_lq_2}) becomes:
\begin{align*}
\dot{y_0}(t) =&\;\; -\frac{[y_0(t) - y_1(t)]\pi_1(t)}{(\pi_0(t) + \pi_1(t))^2}\cdot\left[\pi_1(t)(\theta_A^{-})'\left(\frac{\pi_0(t)}{\pi_0(t) + \pi_1(t)}\right) + \pi_0(t)(\theta_A^{+})'\left(\frac{\pi_1(t)}{\pi_0(t) + \pi_1(t)}\right)\right]\\
&\;\; +[y_0(t) - y_1(t)]\cdot \theta_A^{+}\left(\frac{\pi_1(t)}{\pi_0(t) + \pi_1(t)}\right) - y_2(t)\nu_I\hat a_0(y(t)) - \partial_{\pi_0}c_0(\pi(t))\\
&\;\; - \frac{1}{2}\gamma_I((\hat a_0(y(t)))^2 - \phi_A\left(\frac{\pi_0(t)}{\pi_0(t) + \pi_1(t)}\right) - \phi'_A\left(\frac{\pi_0(t)}{\pi_0(t) + \pi_1(t)}\right) \cdot \frac{\pi_1(t)}{\pi_0(t) + \pi_1(t)},\\\\
\dot{y_1}(t) =&\;\; -\frac{[y_1(t) - y_0(t)]\pi_0(t)}{(\pi_0(t) + \pi_1(t))^2}\cdot\left[\pi_1(t)(\theta_A^{-})'\left(\frac{\pi_0(t)}{\pi_0(t) + \pi_1(t)}\right) + \pi_0(t)(\theta_A^{+})'\left(\frac{\pi_1(t)}{\pi_0(t) + \pi_1(t)}\right)\right]\\
&\;\; +[y_1(t) - y_0(t)]\cdot \theta_A^{+}\left(\frac{\pi_0(t)}{\pi_0(t) + \pi_1(t)}\right) - y_3(t)\nu_H\hat a_1(y(t)) - \partial_{\pi_1}c_0(\pi(t))\\
&\;\; - \frac{1}{2}\gamma_H((\hat a_1(y(t)))^2 - \phi_A\left(\frac{\pi_0(t)}{\pi_0(t) + \pi_1(t)}\right) + \phi'_A\left(\frac{\pi_0(t)}{\pi_0(t) + \pi_1(t)}\right) \cdot \frac{\pi_0(t)}{\pi_0(t) + \pi_1(t)},\\\\
\dot{y_2}(t) =&\;\; -\frac{[y_2(t) - y_3(t)]\pi_3(t)}{(\pi_2(t) + \pi_3(t))^2}\cdot\left[\pi_3(t)(\theta_B^{-})'\left(\frac{\pi_2(t)}{\pi_2(t) + \pi_3(t)}\right) + \pi_2(t)(\theta_B^{+})'\left(\frac{\pi_3(t)}{\pi_2(t) + \pi_3(t)}\right)\right]\\
&\;\; +[y_2(t) - y_3(t)]\cdot \theta_B^{+}\left(\frac{\pi_3(t)}{\pi_2(t) + \pi_3(t)}\right) - y_0(t)\nu_I\hat a_2(y(t)) - \partial_{\pi_2}c_0(\pi(t))\\
&\;\; - \frac{1}{2}\gamma_I((\hat a_2(y(t)))^2 - \phi_B\left(\frac{\pi_2(t)}{\pi_2(t) + \pi_3(t)}\right) - \phi'_B\left(\frac{\pi_2(t)}{\pi_2(t) + \pi_3(t)}\right) \cdot \frac{\pi_3(t)}{\pi_2(t) + \pi_3(t)},\\\\
\dot{y_3}(t) =&\;\; -\frac{[y_3(t) - y_2(t)]\pi_2(t)}{(\pi_2(t) + \pi_3(t))^2}\cdot\left[\pi_3(t)(\theta_B^{-})'\left(\frac{\pi_2(t)}{\pi_2(t) + \pi_3(t)}\right) + \pi_2(t)(\theta_B^{+})'\left(\frac{\pi_3(t)}{\pi_2(t) + \pi_3(t)}\right)\right]\\
&\;\; +[y_3(t) - y_2(t)]\cdot \theta_B^{+}\left(\frac{\pi_2(t)}{\pi_2(t) + \pi_3(t)}\right) - y_1(t)\nu_H\hat a_3(y(t)) - \partial_{\pi_3}c_0(\pi(t))\\
&\;\; - \frac{1}{2}\gamma_H((\hat a_3(y(t)))^2 - \phi_B\left(\frac{\pi_2(t)}{\pi_2(t) + \pi_3(t)}\right) + \phi'_B\left(\frac{\pi_2(t)}{\pi_2(t) + \pi_3(t)}\right) \cdot \frac{\pi_2(t)}{\pi_2(t) + \pi_3(t)},\\\\
\dot{\pi_0}(t)=&\;\;\pi_1(t)\theta_A^-\left(\frac{\pi_0(t)}{\pi_0(t) + \pi_1(t)}\right) - \pi_0(t)\left[\theta_A^+\left(\frac{\pi_1(t)}{\pi_0(t) + \pi_1(t)}\right) + \nu_I\hat a_0(y(t))\right]\\
&\;\; + \pi_2(t)\nu_I\hat a_2(y(t)),\\\\
\dot{\pi_1}(t)=&\;\;\pi_0(t)\theta_A^+\left(\frac{\pi_1(t)}{\pi_0(t) + \pi_1(t)}\right) - \pi_1(t)\left[\theta_A^-\left(\frac{\pi_0(t)}{\pi_0(t) + \pi_1(t)}\right) + \nu_H\hat a_1(y(t))\right]\\
&\;\;+ \pi_3(t)\nu_H\hat a_3(y(t)),\\\\
\dot{\pi_2}(t)=&\;\;\pi_3(t)\theta_B^-\left(\frac{\pi_2(t)}{\pi_2(t) + \pi_3(t)}\right) - \pi_2(t)\left[\theta_B^+\left(\frac{\pi_3(t)}{\pi_2(t) + \pi_3(t)}\right) + \nu_I\hat a_2(y(t))\right]\\
&\;\;+ \pi_0(t)\nu_I\hat a_0(y(t)),\\\\
\dot{\pi_3}(t)=&\;\;\pi_2(t)\theta_B^+\left(\frac{\pi_3(t)}{\pi_2(t) + \pi_3(t)}\right) - \pi_3(t)\left[\theta_B^-\left(\frac{\pi_2(t)}{\pi_2(t) + \pi_3(t)}\right) + \nu_H\hat a_3(y(t))\right]\\
&\;\;+ \pi_1(t)\nu_H\hat a_1(y(t)),\\
\end{align*}
where the optimal control is defined by:
\begin{align*}
\hat a_0(y):=&\;b\left(\frac{\nu_I(y_0 - y_2)}{\gamma_I}\right),\;\;\hat a_1(y):=b\left(\frac{\nu_H(y_1 - y_3)}{\gamma_H}\right),\\
\hat a_2(y):=&\;b\left(\frac{\nu_I(y_2 - y_0)}{\gamma_I}\right),\;\;\hat a_3(y):=b\left(\frac{\nu_H(y_3 - y_1)}{\gamma_H}\right),
\end{align*}
and the terminal conditions are $\pi(0) = \mathbf{p}^{\circ}$ and $y(T) = \nabla C_0(\pi(T))$.

\subsection{Mean field equilibrium in the absence of the authority}

The system of ODEs characterizing the mean field game equilibrium consists of the Hamilton-Jacobi equation and the Kolmogorov equation.
\begin{align*}
\dot{v_0}(t) =&\;\; [v_1(t) - v_0(t)] \theta_A^+\left(\frac{\pi_1(t)}{\pi_0(t) + \pi_1(t)}\right) + [v_2(t) - v_0(t)]\nu_I\hat a_0(v(t))\\
&\;\; + \frac{1}{2}\gamma_I(\hat a_0(v(t)))^2 + \phi_A\left(\frac{\pi_0(t)}{\pi_0(t) + \pi_1(t)}\right),\\\\
\dot{v_1}(t) =&\;\; [v_0(t) - v_1(t)] \theta_A^-\left(\frac{\pi_0(t)}{\pi_0(t) + \pi_1(t)}\right) + [v_3(t) - v_1(t)]\nu_H\hat a_1(v(t))\\
&\;\; + \frac{1}{2}\gamma_H(\hat a_1(v(t)))^2+ \phi_A\left(\frac{\pi_0(t)}{\pi_0(t) + \pi_1(t)}\right),\\\\
\dot{v_2}(t) =&\;\; [v_3(t) - v_2(t)] \theta_B^+\left(\frac{\pi_3(t)}{\pi_2(t) + \pi_3(t)}\right) + [v_0(t) - v_2(t)]\nu_I\hat a_2(v(t))\\
&\;\; + \frac{1}{2}\gamma_I(\hat a_2(v(t)))^2 + \phi_B\left(\frac{\pi_2(t)}{\pi_2(t) + \pi_3(t)}\right),\\\\
\dot{v_3}(t) =&\;\; [v_2(t) - v_3(t)] \theta_B^-\left(\frac{\pi_2(t)}{\pi_2(t) + \pi_3(t)}\right) + [v_1(t) - v_3(t)]\nu_H\hat a_3(v(t))\\
&\;\; + \frac{1}{2}\gamma_H(\hat a_3(v(t)))^2+ \phi_B\left(\frac{\pi_2(t)}{\pi_2(t) + \pi_3(t)}\right),\\\\
\dot{\pi_0}(t)=&\;\;\pi_1(t)\theta_A^-\left(\frac{\pi_0(t)}{\pi_0(t) + \pi_1(t)}\right) - \pi_0(t)\left[\theta_A^+\left(\frac{\pi_1(t)}{\pi_0(t) + \pi_1(t)}\right) + \nu_I\hat a_0(v(t))\right]\\
&\;\;+ \pi_2(t)\nu_I\hat a_2(v(t)),\\\\
\dot{\pi_1}(t)=&\;\;\pi_0(t)\theta_A^+\left(\frac{\pi_1(t)}{\pi_0(t) + \pi_1(t)}\right) - \pi_1(t)\left[\theta_A^-\left(\frac{\pi_0(t)}{\pi_0(t) + \pi_1(t)}\right) + \nu_H\hat a_1(v(t))\right]\\
&\;\;+ \pi_3(t)\nu_H\hat a_3(v(t)),\\\\
\dot{\pi_2}(t)=&\;\;\pi_3(t)\theta_B^-\left(\frac{\pi_2(t)}{\pi_2(t) + \pi_3(t)}\right) - \pi_2(t)\left[\theta_B^+\left(\frac{\pi_3(t)}{\pi_2(t) + \pi_3(t)}\right) + \nu_I\hat a_2(v(t))\right]\\
&\;\;+ \pi_0(t)\nu_I\hat a_0(v(t)),\\\\
\dot{\pi_3}(t)=&\;\;\pi_2(t)\theta_B^+\left(\frac{\pi_3(t)}{\pi_2(t) + \pi_3(t)}\right) - \pi_3(t)\left[\theta_B^-\left(\frac{\pi_2(t)}{\pi_2(t) + \pi_3(t)}\right) + \nu_H\hat a_3(v(t))\right]\\
&\;\;+ \pi_1(t)\nu_H\hat a_1(v(t)),\\
\end{align*}
where the optimal control is defined by:
\begin{align*}
\hat a_0(v):=&\;b\left(\frac{\nu_I(v_0 - v_2)}{\gamma_I}\right),\;\;\hat a_1(v):=b\left(\frac{\nu_H(v_1 - v_3)}{\gamma_H}\right),\\
\hat a_2(v):=&\;b\left(\frac{\nu_I(v_2 - v_0)}{\gamma_I}\right),\;\;\hat a_3(v):=b\left(\frac{\nu_H(v_3 - v_1)}{\gamma_H}\right),
\end{align*}
and the terminal conditions are $\pi(0) = \mathbf{p}^{\circ}$ and $v(T) = 0$.


\bibliographystyle{siam}

\end{document}